\documentclass[a4paper,12pt,reqno]{amsart}
\usepackage[T1]{fontenc}
\usepackage[utf8]{inputenc}
\usepackage{amssymb,dsfont,mathrsfs}
\usepackage[alphabetic]{amsrefs}
\usepackage[margin=1in]{geometry}
\usepackage{enumitem}
\usepackage[bookmarksdepth=2]{hyperref}

\numberwithin{equation}{section}

\newtheorem{theorem}{Theorem}[section]
\newtheorem{lemma}[theorem]{Lemma}
\newtheorem{corollary}[theorem]{Corollary}

\theoremstyle{definition}

\theoremstyle{notation}

\newcommand{\R}{\mathds{R}}
\newcommand{\N}{\mathds{N}}

\renewcommand{\le}{\leqslant}
\renewcommand{\ge}{\geqslant}
\allowdisplaybreaks[1]

\title{Existence theory for a bushfire equation}

\author[S. Dipierro]{Serena Dipierro}
\address{S.D., 
Department of Mathematics and Statistics,
University of Western Australia,
35~Stirling Highway, WA 6009 Crawley, Australia. }

\email{serena.dipierro@uwa.edu.au}

\author[E. Valdinoci]{Enrico Valdinoci}
\address{E.V., 
Department of Mathematics and Statistics,
University of Western Australia,
35~Stirling Highway, WA 6009 Crawley, Australia. }

\email{enrico.valdinoci@uwa.edu.au}

\author[G. Wheeler]{Glen Wheeler}
\address{G. W.,
School of Mathematics and Applied Statistics,
University of Wollongong,
Northfields Avenue, NSW 2500 Wollongong, Australia. }

\email{glenw@uow.edu.au}

\author[V.-M. Wheeler]{Valentina-Mira Wheeler}
\address{V.-M. W.,
School of Mathematics and Applied Statistics,
University of Wollongong,
Northfields Avenue, NSW 2500 Wollongong, Australia. }

\email{vwheeler@uow.edu.au}

\thanks{Supported by Australian Research Council FL190100081 and DE190100379.}
\subjclass{35A01, 47H10, 45H05, 35K57, 76R10.}
\keywords{Bushfire models, existence theory, fixed-point methods, interaction kernels,
pyrogenic flow.}

\begin{document}
\begin{abstract} In~\cite{PAPER1} we have introduced a new model to describe front propagation
in bushfires.
This model describes temperature diffusion in view of an ignition process
induced by an interaction kernel, the effect of the environmental wind and that of
the fire wind.

This has led to the introduction of a new partial differential equation of evolutionary type,
with nonlinear terms both of integral kind and involving the gradient of the solution.
This equation is, in a sense, ``hybrid'', since it encodes both analytical and geometric features of the front propagation. This new characteristic makes the equation particularly interesting
also from the mathematical point of view, often falling outside the territory already covered by standard methods.

In this paper, we start the mathematical treatment of this equation by establishing the short time and global existence theory for this equation.
\end{abstract}
\maketitle

\section{Introduction}

In this paper we consider a model for bushfire that we have recently introduced in~\cite{PAPER1}.
This model describes the spread of the fire on the basis of relatively simple parameters,
namely the environmental temperature, the ignition temperature of the combustible, the environmental wind, and a convective
effect due to the temperature gradient in the vicinity of the flame
(known in the literature as ``pyrogenic flow'').

The model has been confronted with real-world experiments, with numerical simulations compared against data coming from real bushfires in Australia and from lab experiments.

Since this is a new model, the mathematical theory underpinning it is still under development: in this paper we aim at establishing
an existence theory for the partial differential equation at the core of this model.

This partial differential equation is of parabolic type, but it is not completely standard. Therefore, on the one hand, it is tempting to utilize functional analytic techniques, such as fixed-point arguments and iteration methods, to demonstrate the existence of a solution,
but, on the other hand, these methods need to be adapted to the specific model under consideration, which requires some bespoke argument and some ad hoc calculations.\medskip

Our bushfire model aims at describing the evolution of the environmental temperature~$u$ (which will correspond to the solution of the partial differential equation under consideration). When the value of~$u$ is above an effective ignition temperature~$\Theta$, the available fuel burns, producing a forcing term that contributes to the equation through a site interaction described by a kernel~$K$
(roughly speaking, a burning site tends to increase the temperature of its neighbors).\medskip

Two additional effects can be included in our model. First, the environmental wind~$\omega$ can contribute to the spreading of the fire.
Also, an additional effect due to the moving of the air can take part in this complex phenomenon, on the basis of a convective term, which is modulated by a function~$\beta$ (typically, $\beta$ is a function of the temperature, and one can think that it is supported in the vicinity of the ignition temperature, since the convective effects are expected to be more significant in a neighborhood of the flame, where the gradient temperature is likely to be higher). These additional effects accounting for the movement of the air cannot really stop the propagation of a bushfire, instead they can contribute to its spread: for this reason, these types of term are considered as active only when they point ``into the outward direction'' of the front.\medskip

Though a full mathematical description and a precise list of hypotheses will be given in the main results here below,
in virtue of the above discussion, the main equation that we consider takes the form
$$ \partial_t u=\Delta u+
\int_{\Omega}\big(u(y,t)-\Theta(y,t)\big)_+\,K(x,y)\,dy+
\left( \left(\omega+\frac{\beta(u)\nabla u}{|\nabla u|^\gamma}\right)\cdot\nabla u\right)_-\,,$$
and our unknown function is~$u=u(x,t)$.

For concreteness, we focus here on the case of homogeneous Dirichlet boundary conditions (which is physically relevant,
since, for example, Australia is an island and bushfires are not expected to change, at least in a short time, the temperature of the surrounding ocean). While we refer to~\cite{PAPER1} for a comprehensive description of the model, let us mention that the partial differential equation that we consider can also describe the moving front of the fire, corresponding to the level sets of~$u$ that match the ignition temperature. This provides a ``hybrid'' model combining the solution of the partial differential equation under consideration and a geometric term given by the mean curvature of the moving front (and, in this setting, the parabolicity of the partial differential equation entails that the equation of the moving front presents the mean curvature term with ``the right sign'' allowing for short time
smoothing effects).\medskip

Of course, bushfire modeling has a long-standing tradition and several different models
and algorithms have been proposed to deal with different aspects of fire ignition and propagation. Though an exhaustive list of different models would go well beyond the scopes of this article, we refer to the surveys in~\cite{BAK, PERRY, PASTOR} for the description of several models and perspectives, as well as to~\cite{SMITH, HILTON201812} and the reference therein for additional information on the pyrogenic flow.
\medskip

We now dive into the technical description of our main results. The first result that we present deals with short time existence
and goes as follows:

\begin{theorem}[Short time existence for a bushfire equation]\label{MAIN:THEOREM}
Let~$\Omega$ be a bounded domain of~$\R^n$ with~$C^2$ boundary.
Let also~$K\in L^2(\Omega\times\Omega)$, $\Theta\in L^\infty(\Omega\times (0,1))$,
$\omega\in L^\infty(\Omega\times (0,1))$, $\beta\in W^{1,\infty}(\R)$,
and~$\gamma\in(1,2]$.

Then, there exists~$T_\star\in(0,1]$, depending only on~$n$, $\Omega$,
$\|K\|_{L^2(\Omega\times\Omega)}$,
$\|\omega\|_{ L^\infty(\Omega\times(0,1))}$, and~$\|\beta\|_{ W^{1,\infty}(\R)}$, such that if~$T\in(0,T_\star]$ then, given~$g\in H^1_0(\Omega)$, the problem
\begin{equation*}\begin{cases}
\partial_t u(x,t)=\Delta u+\displaystyle
\int_{\Omega}\big(u(y,t)-\Theta(y,t)\big)_+\,K(x,y)\,dy\\
\qquad\qquad\qquad\qquad+\displaystyle
\left( \left(\omega(x,t)+\frac{\beta(u(x,t))\nabla u(x,t)}{|\nabla u(x,t)|^\gamma}\right)\cdot\nabla u(x,t)\right)_- \\ \qquad{\mbox{ for $(x,t)\in\Omega\times(0,T)$,}}\\
\\
u(x,t)=0\\ \qquad{\mbox{ for $(x,t)\in(\partial\Omega)\times(0,T)$,}}\\
\\
u(x,0)=g(x)\\ \qquad{\mbox{ for $x\in\Omega$}}
\end{cases}
\end{equation*}
possesses\footnote{For simplicity, the notion of solution adopted in this paper is the weak one,
see e.g.~\cite[page~352]{MR2597943}.} a solution $u \in L^2((0,T),\,H^1_0(\Omega))$.
\end{theorem}

The proof that we provide of Theorem~\ref{MAIN:THEOREM} also allows for several generalizations:
for instance, one considers the case of a bounded function~$\beta$ also depending on~$x$ and~$t$, and Lipschitz continuous
in its dependence on~$u$, as well as the case of a general second order parabolic operator
(e.g., in the setting of~\cite[Chapter~7]{MR2597943}); here we do not go into all the details of these generalizations,
to keep things as simple as possible.

We also point out that the technical condition~$\gamma\in(1,2]$ corresponds to a requirement on the
relative importance of wind and convection in the equation under consideration:
roughly speaking, while the dependence on the wind is bounded by a term scaling linearly with respect to the gradient of the solution,
the convective term is bounded by a term of the type~$|\nabla u|^{2-\gamma}$, hence sublinear when~$\gamma\in(1,2]$
(stating, in some sense, that the impact of the wind is assumed to be dominant with respect to that of the convection).

The case~$\gamma=1$ in which wind and convection scale the same is also of interest. When~$\beta$ is independent of~$u$
(but possibly depending on space and time, as mentioned above), the proof of Theorem~\ref{MAIN:THEOREM} would go through
also when~$\gamma=1$.

Moreover, under a smallness assumption on the nonlinear terms,
the case of~$\beta$ depending on~$u$ and~$\gamma=1$ may be recovered through a limit procedure, as addressed by the following result:

\begin{theorem}[Short time existence for a bushfire equation with convective terms with linear scaling]\label{MAIN:THEOREM:LIN}
Let~$\Omega$ be a bounded domain of~$\R^n$ with~$C^2$ boundary.
Let also~$K\in L^2(\Omega\times\Omega)$, $\Theta\in L^\infty(\Omega\times (0,1))$,
$\omega\in L^\infty(\Omega\times (0,1))$ and~$\beta\in W^{1,\infty}(\R)$.

Then, there exist~$\epsilon_0\in(0,1)$, depending only on~$n$, $\Omega$ and~$\|\Theta_-\|_{ L^\infty(\Omega\times (0,1))}$, and~$
T_\star\in(0,1]$, depending only on~$n$, $\Omega$,
$\|K\|_{L^2(\Omega\times\Omega)}$,
$\|\omega\|_{ L^\infty(\Omega\times(0,1))}$, and~$\|\beta\|_{ W^{1,\infty}(\R)}$, such that if
\begin{equation*}
\|K\|_{L^2(\Omega\times\Omega)}+\|\omega\|_{ L^\infty(\Omega\times(0,1))}+\|\beta\|_{ L^{\infty}(\R)}\le\epsilon_0
\end{equation*}
and~$T\in(0,T_\star]$ then, given~$g\in H^1_0(\Omega)$, the problem
\begin{equation*}\begin{cases}
\partial_t u(x,t)=\Delta u+\displaystyle
\int_{\Omega}\big(u(y,t)-\Theta(y,t)\big)_+\,K(x,y)\,dy\\
\qquad\qquad\qquad\qquad+\displaystyle
\left( \left(\omega(x,t)+\frac{\beta(u(x,t))\nabla u(x,t)}{|\nabla u(x,t)|}\right)\cdot\nabla u(x,t)\right)_- \\ \qquad{\mbox{ for $(x,t)\in\Omega\times(0,T)$,}}\\
\\
u(x,t)=0\\ \qquad{\mbox{ for $(x,t)\in(\partial\Omega)\times(0,T)$,}}\\
\\
u(x,0)=g(x)\\ \qquad{\mbox{ for $x\in\Omega$}}
\end{cases}
\end{equation*}
possesses a solution $u \in L^2((0,T),\,H^1_0(\Omega))$.
\end{theorem}

{F}rom Theorem~\ref{MAIN:THEOREM}, under some uniform assumption on the data, one can also obtain a global in time solution, according to the following result:

\begin{theorem}[Global in time existence for a bushfire equation]\label{MAIN:THEOREM:GLO}
Let~$\Omega$ be a bounded domain of~$\R^n$ with~$C^2$ boundary.
Let also~$K\in L^2(\Omega\times\Omega)$, $\Theta\in L^\infty(\Omega\times (0,+\infty))$,
$\omega\in L^\infty(\Omega\times (0,+\infty))$, $\beta\in W^{1,\infty}(\R)$,
and~$\gamma\in(1,2]$.

Then, given~$g\in H^1_0(\Omega)$, the problem
\begin{equation*}\begin{cases}
\partial_t u(x,t)=\Delta u+\displaystyle
\int_{\Omega}\big(u(y,t)-\Theta(y,t)\big)_+\,K(x,y)\,dy\\
\qquad\qquad\qquad\qquad+\displaystyle
\left( \left(\omega(x,t)+\frac{\beta(u(x,t))\nabla u(x,t)}{|\nabla u(x,t)|^\gamma}\right)\cdot\nabla u(x,t)\right)_- \\ \qquad{\mbox{ for $(x,t)\in\Omega\times(0,+\infty)$,}}\\
\\
u(x,t)=0\\ \qquad{\mbox{ for $(x,t)\in(\partial\Omega)\times(0,+\infty)$,}}\\
\\
u(x,0)=g(x)\\ \qquad{\mbox{ for $x\in\Omega$}}
\end{cases}
\end{equation*}
possesses a solution $u \in L^2((0,+\infty),\,H^1_0(\Omega))$.
\end{theorem}

For completeness, let us mention that a global in time result similar to Theorem~\ref{MAIN:THEOREM:GLO} holds true, with the same proof,
when~$\gamma=1$, under a smaller condition on the structural parameter (this would follow by leveraging Theorem~\ref{MAIN:THEOREM:LIN} in lieu of
Theorem~\ref{MAIN:THEOREM}).\medskip

The rest of the paper is devoted to the proof of the main results stated above.
Namely, the short time existence result in Theorem~\ref{MAIN:THEOREM} will be proved
in Section~\ref{S:LECADDF} and its variant to comprise the case of linearly scaling convective terms
will be considered in Section~\ref{PWSJM-0-i4rpo2438bv.23tno06u65830-1},
while Section~\ref{S:LECADDF2} will be devoted to the proof of Theorem~\ref{MAIN:THEOREM:GLO}.

\section{Proof of Theorem~\ref{MAIN:THEOREM}}\label{S:LECADDF}

\subsection{Toolbox}

This section collects a number of technical results which will come in handy in Section~\ref{FA6S} to provide an efficient setup for the fixed-point argument.

\begin{lemma}\label{wHR}
Let~$\Omega$ be a bounded, Lipschitz domain of~$\R^n$.
Let also~$\beta\in W^{1,\infty}(\R)$ and~$\gamma\in(1,2]$.

Then, there exists a constant~$C>0$, depending only on~$n$, $\Omega$, and~$\|\beta\|_{W^{1,\infty}(\R)}$, such that, for every~$u$, $v\in H^1_0(\Omega)$,
\begin{equation}\label{DEN5} \int_\Omega |\beta(u(x))-\beta(v(x))|^2 |\nabla u(x)|^{2(2-\gamma)}\,dx\le
C\,\|\nabla u\|_{L^2(\Omega)}^{2(2-\gamma)}
\|u-v\|_{L^2(\Omega)}^{2(\gamma-1)}.\end{equation}
\end{lemma}

\begin{proof}
We use the H\"older Inequality with exponents~$\frac{1}{\gamma-1}$ and~$\frac{1}{2-\gamma}$ to see that
\begin{equation}\label{DEN3} \begin{split}&\int_\Omega |\beta(u(x))-\beta(v(x))|^2 |\nabla u(x)|^{2(2-\gamma)}\,dx\\&\qquad\le
\left(\int_\Omega |\beta(u(x))-\beta(v(x))|^{\frac{2}{\gamma-1}}\,dx\right)^{\gamma-1}
\left(\int_\Omega |\nabla u(x)|^2\,dx\right)^{2-\gamma}.\end{split}\end{equation}
Furthermore,
\begin{eqnarray*}
&& |\beta(u(x))-\beta(v(x))|^{\frac{2}{\gamma-1}}=
|\beta(u(x))-\beta(v(x))|^{\frac{4-2\gamma}{\gamma-1}}
|\beta(u(x))-\beta(v(x))|^{2}\\&&\qquad\le\Big(
|\beta(u(x))|+|\beta(v(x))|\Big)^{\frac{4-2\gamma}{\gamma-1}}
\|\nabla\beta\|_{L^\infty(\R)}^2|u(x)-v(x)|^2\\&&\qquad
\le 2^{\frac{4-2\gamma}{\gamma-1}}\|\beta\|_{L^\infty(\R)}^{\frac{4-2\gamma}{\gamma-1}}\|\nabla\beta\|_{L^\infty(\R)}^2|u(x)-v(x)|^2.
\end{eqnarray*}
This and~\eqref{DEN3} give~\eqref{DEN5}, as desired.
\end{proof}

\begin{lemma}\label{LEa-1}
Let~$\Theta\in L^\infty(\Omega)$, $K\in L^2(\Omega\times\Omega)$, $u\in L^2(\Omega)$, and
\begin{equation}\label{f1u} f_{1,u}(x):= \int_{\Omega}\big(u(y)-\Theta(y)\big)_+\,K(x,y)\,dy.\end{equation}

Then,
\begin{equation*} \|f_{1,u}\|^2_{L^2(\Omega)}\le2 \|K\|^2_{L^2(\Omega\times\Omega)}\,\left(\|u\|_{L^2(\Omega)}^2+
|\Omega|\sup_\Omega \Theta_-^2
\right),\end{equation*}
where~$|\Omega|$ is the Lebesgue measure of~$\Omega$.
\end{lemma}

\begin{proof} If~$u(y)\ge\Theta(y)$ then
$$ 0\le \big(u(y)-\Theta(y)\big)_+= u(y)-\Theta(y)\le u(y)+\sup_\Omega \Theta_-.$$
As a result, for all~$y\in\Omega$,
$$ \big(u(y)-\Theta(y)\big)_+^2\le \left(u(y)+\sup_\Omega \Theta_-\right)^2\le
2\left(u^2(y)+\sup_\Omega \Theta_-^2\right)$$
and, that being so, using the Cauchy-Schwarz Inequality we obtain that
\begin{eqnarray*}&& \left(\int_{\Omega}\big(u(y)-\Theta(y)\big)_+\,K(x,y)\,dy\right)^2\le
\int_{\Omega}\big(u(y)-\Theta(y)\big)_+^2\,dy\,\int_\Omega K^2(x,y)\,dy\\&&\qquad
\le2\left(\int_{\Omega}u^2(y)\,dy+|\Omega|\sup_\Omega \Theta_-^2\right)
\,\int_\Omega K^2(x,y)\,dy.\end{eqnarray*}
Integrating for~$x\in\Omega$, the desired result follows.
\end{proof}

\begin{lemma}\label{LEa-2}
Let~$\omega\in L^\infty(\Omega)$, $\beta\in L^\infty(\R)$, $\gamma\in(1,2]$, and~$u\in H^1(\Omega)$.

Let also
\begin{equation}\label{df2} f_{2,u}(x):=\left( \left(\omega(x)+\frac{\beta(u(x))\nabla u(x)}{|\nabla u(x)|^\gamma}\right)\cdot\nabla u(x)\right)_-\,.\end{equation}

Then,
$$ \|f_{2,u}\|_{L^2(\Omega)}\le \sqrt{2}\left(
\|\omega\|_{L^\infty(\Omega)}\,\|\nabla u\|_{L^2(\Omega)}
+\|\beta\|_{L^\infty(\R)}\,|\Omega|^{\frac{\gamma-1}{2}}
\|\nabla u\|_{L^2(\Omega)}^{2-\gamma}\right).
$$
\end{lemma}

\begin{proof} We have that
\begin{eqnarray*}
| f_{2,u}(x)|&\le& |\omega(x)|\,|\nabla u(x)|
+|\beta(u(x))|\,|\nabla u(x)|^{2-\gamma}\\&\le&
\|\omega\|_{L^\infty(\Omega)}\,|\nabla u(x)|
+\|\beta\|_{L^\infty(\R)}\,|\nabla u(x)|^{2-\gamma}.
\end{eqnarray*}
Therefore, squaring and integrating over~$\Omega$,
\begin{equation}\label{FAG3}
\| f_{2,u}\|_{L^2(\Omega)}^2\le 2\left(
\|\omega\|_{L^\infty(\Omega)}^2\,\|\nabla u(x)\|_{L^2(\Omega)}^2
+\|\beta\|_{L^\infty(\R)}^2\,\int_\Omega|\nabla u(x)|^{2(2-\gamma)}\,dx\right).
\end{equation}

Now, if~$\gamma=2$, from~\eqref{FAG3} we obtain that
\begin{equation}\label{FAG4}
\| f_{2,u}\|_{L^2(\Omega)}\le \sqrt{2}\left(
\|\omega\|_{L^\infty(\Omega)}\,\|\nabla u(x)\|_{L^2(\Omega)}
+\|\beta\|_{L^\infty(\R)}\,\sqrt{|\Omega|}\right).
\end{equation}

If instead~$\gamma\in(1,2)$, we 
exploit the H\"older Inequality with exponents~$\frac1{2-\gamma}$ and~$\frac1{\gamma-1}$ to see that
\begin{eqnarray*}
\int_\Omega|\nabla u(x)|^{2(2-\gamma)}\,dx\le |\Omega|^{\gamma-1}
\left(\int_\Omega|\nabla u(x)|^{2}\,dx\right)^{2-\gamma}
=|\Omega|^{\gamma-1}\|\nabla u\|_{L^2(\Omega)}^{2(2-\gamma)}
.\end{eqnarray*}
Plugging this information into~\eqref{FAG3} we conclude that, in this case,
\begin{equation*}
\| f_{2,u}\|_{L^2(\Omega)}\le \sqrt{2}\left(
\|\omega\|_{L^\infty(\Omega)}\,\|\nabla u(x)\|_{L^2(\Omega)}
+\|\beta\|_{L^\infty(\R)}\,|\Omega|^{\frac{\gamma-1}{2}}
\|\nabla u\|_{L^2(\Omega)}^{2-\gamma}\right).
\end{equation*}
{F}rom this and~\eqref{FAG4} the desired result follows.
\end{proof}

\begin{lemma}\label{L14}
Let~$\Theta\in L^\infty(\Omega)$, $K\in L^2(\Omega\times\Omega)$, and~$u$, $v \in L^2(\Omega)$.

Then, in the notation of~\eqref{f1u},
\begin{equation*} \|f_{1,u}-f_{1,v}\|_{L^2(\Omega)}\le\|K\|_{L^2(\Omega\times\Omega)}
\|u-v\|_{L^2(\Omega)}.\end{equation*}
\end{lemma}

\begin{proof} Since the positive part is a Lipschitz function, with Lipschitz constant controlled by~$1$, we have that, for all~$a$, $b\in\R$,
$$ | a_+ - b_+|\le|a-b|.$$
In particular, taking~$a:=u(y)-\Theta(y)$ and~$b:=v(y)-\Theta(y)$, we see that
$$ \Big| \big(u(y)-\Theta(y)\big)_+-\big(v(y)-\Theta(y)\big)_+\Big|\le|u(y)-v(y)|.$$
On this account,
\begin{eqnarray*}
|f_{1,u}(x)-f_{1,v}(x)|&=&\left| \int_{\Omega}\Big(\big(u(y)-\Theta(y)\big)_+-\big(v(y)-\Theta(y)\big)_+\Big)
\,K(x,y)\,dy\right|\\&\le&
\int_{\Omega}\Big|\big(u(y)-\Theta(y)\big)_+-\big(v(y)-\Theta(y)\big)_+\Big|
\,|K(x,y)|\,dy\\&\le&
\int_{\Omega}|u(y)-v(y)|\,|K(x,y)|\,dy.
\end{eqnarray*}
Using the H\"older Inequality, we conclude that
$$ |f_{1,u}(x)-f_{1,v}(x)|^2\le
\int_{\Omega}|u(y)-v(y)|^2\,dy
\int_{\Omega}K^2(x,y)\,dy$$
and the desired result follows by integrating for~$x\in\Omega$.
\end{proof}

\begin{lemma}\label{pklsmd}
Let~$\Omega$ be a bounded, Lipschitz domain of~$\R^n$.
Let~$\omega\in L^\infty(\Omega)$, $\beta\in W^{1,\infty}(\R)$, and~$\gamma\in(1,2]$.

Then, there exist a constant~$C>0$, depending only on~$n$, $\Omega$, $\|\omega\|_{L^\infty(\Omega)}$, and~$\|\beta\|_{W^{1,\infty}(\R)}$, such that, 
in the notation of~\eqref{df2}, for every~$u$, $v\in H^1_0(\Omega)$ it holds that
$$ \|f_{2,u}-f_{2,v}\|_{L^2(\Omega)}\le\begin{cases}
C \Big(\|\nabla u-\nabla v\|_{L^2(\Omega)}
+\|\nabla u-\nabla v\|_{L^2(\Omega)}^{2-\gamma}
+\|\nabla u\|_{L^2(\Omega)}^{2-\gamma}\,\|u-v\|_{L^2(\Omega)}^{\gamma-1}\Big)\\
\qquad\qquad\qquad  {\mbox{ if~$\gamma\in(1,2)$,}}\\ 
C  \|u-v\|_{H^1(\Omega)}  \quad {\mbox{ if~$\gamma=2$.}}
\end{cases}$$
\end{lemma}

\begin{proof} We observe that
the negative part is a Lipschitz function, with Lipschitz constant controlled by~$1$, whence, for all~$a$, $b\in\R$,
$$ | a_- - b_-|\le|a-b|.$$
In particular, taking
\begin{eqnarray*}
&& a:=\left(\omega(x)+\frac{\beta(u(x))\nabla u(x)}{|\nabla u(x)|^\gamma}\right)\cdot\nabla u(x)\\
{\mbox{and }}&&b:=\left(\omega(x)+\frac{\beta(v(x))\nabla v(x)}{|\nabla v(x)|^\gamma}\right)\cdot\nabla v(x),
\end{eqnarray*}
we find that
\begin{equation*} \begin{split}&
|f_{2,u}(x)-f_{2,v}(x)|\\&\quad=
\left|
\left( \left(\omega(x)+\frac{\beta(u(x))\nabla u(x)}{|\nabla u(x)|^\gamma}\right)\cdot\nabla u(x)\right)_-
-\left(\left(\omega(x)+\frac{\beta(v(x))\nabla v(x)}{|\nabla v(x)|^\gamma}\right)\cdot\nabla v(x)\right)_-\right|\\
&\quad\le
\left|
\left(\omega(x)+\frac{\beta(u(x))\nabla u(x)}{|\nabla u(x)|^\gamma}\right)\cdot\nabla u(x)
-\left(\omega(x)+\frac{\beta(v(x))\nabla v(x)}{|\nabla v(x)|^\gamma}\right)\cdot\nabla v(x)\right|\\
&\quad\le |\omega(x)|\,|\nabla u(x)-\nabla v(x)|+
\Big| \beta(u(x))|\nabla u(x)|^{2-\gamma}-\beta(v(x))|\nabla v(x)|^{2-\gamma}\Big|\\&\quad\le\|\omega\|_{L^\infty(\Omega)}|\nabla u(x)-\nabla v(x)|\\&\qquad
+|\beta(v(x))|\,\Big| |\nabla u(x)|^{2-\gamma}-|\nabla v(x)|^{2-\gamma}\Big|+
\Big| \big(\beta(u(x))-\beta(v(x))\big)|\nabla u(x)|^{2-\gamma}\Big|
\\&\quad\le\|\omega\|_{L^\infty(\Omega)}|\nabla u(x)-\nabla v(x)|\\&\qquad
+\|\beta\|_{L^\infty(\R)}\,\Big| |\nabla u(x)|^{2-\gamma}-|\nabla v(x)|^{2-\gamma}\Big|+\big|\beta(u(x))-\beta(v(x))\big|\,|\nabla u(x)|^{2-\gamma}
.\end{split}\end{equation*}
Notice that if~$\gamma=2$ then~$\Big| |\nabla u(x)|^{2-\gamma}-|\nabla v(x)|^{2-\gamma}\Big|=0$.
Hence, from now on we suppose that the second term in the last inequality above is present only when~$\gamma \in(1,2)$, otherwise
it is set to be zero in the following computations.

Accordingly,
\begin{equation} \label{POI86}\begin{split}&
\int_\Omega|f_{2,u}(x)-f_{2,v}(x)|^2\,dx
\\ \le\;&3\|\omega\|_{L^\infty(\Omega)}^2
\int_\Omega|\nabla u(x)-\nabla v(x)|^2\,dx
+3\|\beta\|_{L^\infty(\R)}^2\int_\Omega 
\Big| |\nabla u(x)|^{2-\gamma}-|\nabla v(x)|^{2-\gamma}\Big|^2\,dx
\\&\qquad+
3\int_\Omega \big|\beta(u(x))-\beta(v(x))\big|^2\,|\nabla u(x)|^{2(2-\gamma)}\,dx.
\end{split}\end{equation}

We observe that, for all~$a$, $b>0$,
\begin{equation}\label{POI87}
\big| a^{2-\gamma}- b^{2-\gamma}\big|\le |a-b|^{2-\gamma}.
\end{equation}
To check this, we consider the function~$t\in(0,1)\mapsto g(t):=\frac{1-t^{2-\gamma}}{(1-t)^{2-\gamma}}$ and we point out that~$g(0)=1$ and
\begin{eqnarray*}
&& \lim_{t\to1}g(t)=\lim_{t\to1}\frac{1-t^{2-\gamma}}{(1-t)^{2-\gamma}}
=\lim_{t\to1}\frac{-(2-\gamma)t^{1-\gamma}}{(2-\gamma)(1-t)^{1-\gamma}}=0.
\end{eqnarray*}
Moreover,
\begin{eqnarray*}
g'(t)&=&\frac{-(2-\gamma)t^{1-\gamma}(1-t)^{2-\gamma}+(2-\gamma)(1-t^{2-\gamma})(1-t)^{1-\gamma}}{ (1-t)^{2(2-\gamma)} }\\
&=&\frac{(2-\gamma)(1-t^{1-\gamma})}{(1-t)^{3-\gamma}}\\
&\le& 0.
\end{eqnarray*}
Gathering these pieces of information,
we obtain that~$g(t)\le1$ for all~$t\in(0,1)$.
Namely, for all~$t\in(0,1)$,
\begin{equation}\label{LPO1}
1-t^{2-\gamma}\le (1-t)^{2-\gamma}.\end{equation}

Now, we suppose, without loss of generality, that~$a> b$ and we
write
$$ \big| a^{2-\gamma}- b^{2-\gamma}\big|=
a^{2-\gamma}- b^{2-\gamma}= a^{2-\gamma}\left(1- t^{2-\gamma}\right),$$
with~$t:=b/a\in(0,1)$. Therefore, we deduce from~\eqref{LPO1} that
$$ \big| a^{2-\gamma}- b^{2-\gamma}\big|\le a^{2-\gamma}
(1-t)^{2-\gamma}=(a-b)^{2-\gamma},$$
which establishes~\eqref{POI87}.

Using~\eqref{POI87} with~$a:=|\nabla u(x)|$ and~$b:=|\nabla v(x)|$, we conclude that
\begin{equation}\label{Pkasmx-162twgdk10edofhigbvuwefgh032iweo}
\Big| |\nabla u(x)|^{2-\gamma}-|\nabla v(x)|^{2-\gamma}\Big|^2\le 
\Big| |\nabla u(x)|-|\nabla v(x)|\Big|^{2(2-\gamma)}\le \Big| \nabla u(x)-\nabla v(x)\Big|^{2(2-\gamma)}
\end{equation}
In light of this
and Lemma~\ref{wHR}, we deduce from~\eqref{POI86} that
\begin{equation}\label{POI88}\begin{split}&
\|f_{2,u}-f_{2,v}\|^2_{L^2(\Omega)} \\
\le\;& C\Bigg[ \|\omega\|_{L^\infty(\Omega)}^2
\|\nabla u-\nabla v\|_{L^2(\Omega)}^2
+\|\beta\|_{L^\infty(\R)}^2\int_\Omega 
\Big| \nabla u(x)-\nabla v(x)\Big|^{2(2-\gamma)}\,dx\\&\qquad+\|\nabla u\|_{L^2(\Omega)}^{2(2-\gamma)}
\|u-v\|_{L^2(\Omega)}^{2(\gamma-1)}\Bigg].
\end{split}\end{equation}

Now we exploit the H\"older Inequality with exponents~$\frac1{2-\gamma}$ and~$\frac1{\gamma-1}$ and see that
$$\int_\Omega \Big| \nabla u(x)-\nabla v(x)\Big|^{2(2-\gamma)}\,dx\le|\Omega|^{\gamma-1}
\left(\int_\Omega \Big| \nabla u(x)-\nabla v(x)\Big|^{2}\,dx\right)^{2-\gamma}.
$$
Plugging this information into~\eqref{POI88} entails the desired result, up to renaming constants.
\end{proof}

\subsection{The fixed-point argument}\label{FA6S} 

Here we consider a bounded domain~$\Omega\subset\R^n$ with Lipschitz boundary, we let~$T>0$, and we consider the functional spaces
$$ X:= L^2((0,T),\,H^1_0(\Omega))\qquad{\mbox{and}}\qquad
Y:=L^2((0,T),\,L^2(\Omega)).$$
See e.g.~\cite{MR916688} and the references therein for classical results about  the theory of Sobolev spaces of Hilbert and Banach space-valued functions.

Also, we take~$g\in H^1_0(\Omega)$.

For every~$f\in Y$, we define~$\sigma_f=\sigma_f(x,t)$ to be the solution of the parabolic problem
\begin{equation*}\begin{cases}
\partial_t \sigma_f(x,t)=\Delta\sigma_f(x,t)+f(x,t) & {\mbox{ for $(x,t)\in\Omega\times(0,T)$,}}\\
\sigma_f(x,t)=0& {\mbox{ for $(x,t)\in(\partial\Omega)\times(0,T)$,}}\\
\sigma_f(x,0)=g(x)& {\mbox{ for $x\in\Omega$.}}
\end{cases}
\end{equation*}
We recall that the above solution indeed exists (say, as a weak solution,
see Theorem~3 on page~356 of~\cite{MR2597943}) and it is unique (see
Theorem~4 on page~358 of~\cite{MR2597943}).

Additionally (see Theorem~5 on page~360 of~\cite{MR2597943}) we have the regularity estimate
\begin{equation}\label{REGES}\begin{split}&
\sup_{t\in[0,T]}\|\sigma_f(\cdot,t)\|_{H^1(\Omega)}+\|\sigma_f\|_{L^2((0,T),\,H^2(\Omega))}+\|\partial_t \sigma_f\|_{L^2((0,T),\,L^2(\Omega))}\\&\qquad\le C\Big(\|f\|_{L^2((0,T),\,L^2(\Omega))}+\|g\|_{H^1(\Omega)}\Big),\end{split}
\end{equation}
for some~$C>0$ depending on~$n$, $\Omega$, and~$T$.\medskip

We remark that:

\begin{lemma}\label{LE21.1}
When~$T\in(0,1]$ the constant~$C$ in~\eqref{REGES}
depends on~$n$ and~$\Omega$, but not on~$T$. \end{lemma}

\begin{proof} We define~$\widetilde f(x,t):=f(x,t)\chi_{[0,T]}(t)$ and notice that, by the uniqueness of the solution and the fact that~$\widetilde f=f$ when~$t\in[0,T]$, we have~$\sigma_{\widetilde f}=\sigma_f$ when~$t\in[0,T)$.
We also observe that, for all~$T\in(0,1]$,
$$ \|\widetilde f\|_{L^2((0,1),\,L^2(\Omega))}^2=
\iint_{\Omega\times(0,1)}|\widetilde f(x,t)|^2\,dx\,dt=
\iint_{\Omega\times(0,T)}|f(x,t)|^2\,dx\,dt=\| f\|_{L^2((0,T),\,L^2(\Omega))}^2.$$
As a consequence, we can apply~\eqref{REGES} for~$T:=1$ and obtain that \begin{equation*}
\begin{split} &\sup_{t\in[0,1]}\|\sigma_{\widetilde f}(\cdot,t)\|_{H^1(\Omega)}+\|\sigma_{\widetilde f}\|_{L^2((0,1),\,H^2(\Omega))}+\|\partial_t \sigma_{\widetilde f}\|_{L^2((0,1),\,L^2(\Omega))}\\
&\qquad\le C\Big(\|\widetilde f\|_{L^2((0,1),\,L^2(\Omega))}+\|g\|_{H^1(\Omega)}\Big)
\\&\qquad= C\Big(\|f\|_{L^2((0,T),\,L^2(\Omega))}+\|g\|_{H^1(\Omega)}\Big)
,\end{split}\end{equation*}
with~$C>0$ depending only on~$n$ and~$\Omega$.

This and the fact that, for all~$T\in(0,1]$,
\begin{eqnarray*}
&&\sup_{t\in[0,T]}\|\sigma_{ f}(\cdot,t)\|_{H^1(\Omega)}+\|\sigma_{ f}\|_{L^2((0,T),\,H^2(\Omega))}+\|\partial_t \sigma_{ f}\|_{L^2((0,T),\,L^2(\Omega))}\\
&&\qquad=\sup_{t\in[0,T]}\|\sigma_{\widetilde f}(\cdot,t)\|_{H^1(\Omega)}+\|\sigma_{\widetilde f}\|_{L^2((0,T),\,H^2(\Omega))}+\|\partial_t \sigma_{\widetilde f}\|_{L^2((0,T),\,L^2(\Omega))}\\&&\qquad\le
\sup_{t\in[0,1]}\|\sigma_{\widetilde f}(\cdot,t)\|_{H^1(\Omega)}+\|\sigma_{\widetilde f}\|_{L^2((0,1),\,H^2(\Omega))}+\|\partial_t \sigma_{\widetilde f}\|_{L^2((0,1),\,L^2(\Omega))}
\end{eqnarray*}
give the desired result.
\end{proof}

Now we consider functions~$K\in L^2(\Omega\times\Omega)$, $\Theta\in L^\infty(\Omega\times(0,T))$,
$\omega\in L^\infty(\Omega\times(0,T))$, and~$\beta\in W^{1,\infty}(\R)$. We also let~$\gamma\in(1,2]$.

Given~$u\in X$, we define
\begin{equation}\label{defu}\begin{split} f_u(x,t)&:=\int_{\R^n}\big(u(y,t)-\Theta(y,t)\big)_+\,K(x,y)\,dy\\&\qquad\qquad+\left( \left(\omega(x,t)+\frac{\beta(u(x,t))\nabla u(x,t)}{|\nabla u(x,t)|^\gamma}\right)\cdot\nabla u(x,t)\right)_-.
\end{split}\end{equation}

In this setting, we have:

\begin{lemma}\label{Lasf}
It holds that~$f_u\in Y$, with
$$\|f_u\|_Y\le C\,\left(\|u\|_X+\|u\|_X^{2-\gamma}+T\sup_{\Omega\times(0,T)} \Theta_-\right),$$ for a suitable~$C>0$ depending only on~$n$, $\Omega$, $\|K\|_{L^2(\Omega\times\Omega)}$,
$\|\omega\|_{ L^\infty(\Omega\times(0,T))}$, and~$\|\beta\|_{ W^{1,\infty}(\R)}$.
\end{lemma}

\begin{proof} For a given~$t\in(0,T)$, after~\eqref{f1u} and~\eqref{df2}, we have that
\begin{equation}\label{ANmas} f_u(x,t)=f_{1,u^{(t)}}(x)+f_{2,u^{(t)}}(x),\qquad{\mbox{where}}\qquad {u^{(t)}}(x):=u(x,t).\end{equation}

By Lemmata~\ref{LEa-1} and~\ref{LEa-2}, for any~$t\in(0,T)$,
\begin{eqnarray*} \|f_u(\cdot,t)\|_{L^2(\Omega)}&\le&\|f_{1,u^{(t)}}\|_{L^2(\Omega)}+\|f_{2,u^{(t)}}\|_{L^2(\Omega)}
\\&\le& C\left(
\|u^{(t)}\|_{H^1(\Omega)}+\|u^{(t)}\|_{H^1(\Omega)}^{2-\gamma}+
\sup_{x\in\Omega} \Theta_-(x,t)
\right),\end{eqnarray*}
whence the desired result follows by squaring and integrating for~$t\in(0,T)$.
\end{proof}

Now we consider the nonlinear operator
$$ X\ni u \longmapsto Au:=\sigma_{f_u}.$$

We have that:

\begin{lemma} \label{LA23}
For all~$u\in X$, we have that~$Au\in X$, with
\begin{equation}\label{Lasf.01} \|Au\|_X\le C\left(\|u\|_X+\|u\|_X^{2-\gamma}+\|g\|_{H^1(\Omega)}+T\sup_{\Omega\times(0,T)} \Theta_-\right).\end{equation}

Moreover, for all~$u$, $v\in X$ with~$\|u-v\|_X\le1$,
\begin{equation}\label{Lasf.02} \|Au-Av\|_X\le C_{u}\,\|u-v\|_{X}^{\alpha}.\end{equation}

Here above, the constants~$C>0$ and~$\alpha$ depend only on~$n$, $\Omega$, $T$, $\|K\|_{L^2(\Omega\times\Omega)}$,
$\|\omega\|_{ L^\infty(\Omega\times(0,T))}$, and~$\|\beta\|_{ W^{1,\infty}(\R)}$,
while~$C_{u}>0$ depends also on~$\|u\|_{X}$.
\end{lemma}

\begin{proof} Owing to~\eqref{REGES},
\begin{eqnarray*}&&
\|Au\|_X=\| \sigma_{f_u}\|_{L^2((0,T),\,H^1(\Omega))}
\le C\Big(\|f_u\|_{L^2((0,T),\,L^2(\Omega))}+\|g\|_{H^1(\Omega)}\Big)\\&&\qquad\qquad\qquad= C\Big(\|f_u\|_{Y}+\|g\|_{H^1(\Omega)}\Big).
\end{eqnarray*}
This and Lemma~\ref{Lasf} yield~\eqref{Lasf.01}, as desired.

Now we let~$\sigma_\star:=\sigma_{f_u}-\sigma_{f_v}$ and~$f_\star(x,t):=f_u(x,t)-f_v(x,t)$, and we see that
\begin{equation*}\begin{cases}
\partial_t \sigma_\star(x,t)=\Delta\sigma_\star(x,t)+f_\star(x,t) & {\mbox{ for $(x,t)\in\Omega\times(0,T)$,}}\\
\sigma_\star(x,t)=0& {\mbox{ for $(x,t)\in(\partial\Omega)\times(0,T)$,}}\\
\sigma_\star(x,0)=0& {\mbox{ for $x\in\Omega$.}}
\end{cases}
\end{equation*}
For this reason, by~\eqref{REGES},
\begin{eqnarray*}
\|\sigma_\star\|_{L^2((0,T),\,H^2(\Omega))}\le C \|f_\star\|_{L^2((0,T),\,L^2(\Omega))}=
C\Big(
\|f_u-f_v\|_{L^2((0,T),\,L^2(\Omega))}
\Big)
\end{eqnarray*}
and therefore
\begin{equation}\label{pol-1}\begin{split}&
\|Au-Av\|_X=\|\sigma_{f_u}-\sigma_{f_v}\|_{L^2((0,T),\,H^1(\Omega))}=\|\sigma_\star\|_{L^2((0,T),\,H^1(\Omega))}\\&\qquad\le C\Big(
\|f_u-f_v\|_{L^2((0,T),\,L^2(\Omega))}
\Big).\end{split}
\end{equation}

Now, in the notation of~\eqref{f1u}, \eqref{df2} and~\eqref{ANmas}, and recalling Lemmata~\ref{L14} and~\ref{pklsmd}, we see that, for every~$t\in(0,T)$,
\begin{equation*}\begin{split}&
\| f_u(\cdot,t)-f_v(\cdot,t)\|_{L^2(\Omega)}\\&\quad=
\|f_{1,u^{(t)}}+f_{2,u^{(t)}}-f_{1,v^{(t)}}-f_{2,v^{(t)}}\|_{L^2(\Omega)}\\
&\quad\le
\|f_{1,u^{(t)}}-f_{1,v^{(t)}}\|_{L^2(\Omega)}+\|f_{2,u^{(t)}}-f_{2,v^{(t)}}\|_{L^2(\Omega)}
\\&\quad
\le C \Big(\|\nabla u^{(t)}\|_{L^2(\Omega)}^{2-\gamma}+1\Big)\Big(
\|u^{(t)}-v^{(t)}\|_{H^1(\Omega)}+\|u^{(t)}-v^{(t)}\|_{L^2(\Omega)}^{\gamma-1}+
\|\nabla u^{(t)}-\nabla v^{(t)}\|_{L^2(\Omega)}^{2-\gamma}
\Big).
\end{split}\end{equation*}
We point out that the last term in the above inequality is set to equal zero when~$\gamma=2$, as specified in Lemma~\ref{pklsmd}.

Hence, for all~$\epsilon>0$,
\begin{equation*}\begin{split}&
\| f_u(\cdot,t)-f_v(\cdot,t)\|_{L^2(\Omega)}^2\\&\qquad\le 
C\epsilon \Big(\|\nabla u^{(t)}\|_{L^2(\Omega)}^{2(2-\gamma)}+1\Big)
\\&\qquad\qquad+\frac{C}\epsilon\Big(\|u^{(t)}-v^{(t)}\|_{H^1(\Omega)}^2+\|u^{(t)}-v^{(t)}\|_{L^2(\Omega)}^{2(\gamma-1)}+
\|\nabla u^{(t)}-\nabla v^{(t)}\|_{L^2(\Omega)}^{2(2-\gamma)}\Big)\\
&\qquad\le 
C\epsilon \Big(\|\nabla u^{(t)}\|_{L^2(\Omega)}^{2}+1\Big)
\\&\qquad\qquad+\frac{C}\epsilon\Big(\|u^{(t)}-v^{(t)}\|_{H^1(\Omega)}^2+\|u^{(t)}-v^{(t)}\|_{L^2(\Omega)}^{2(\gamma-1)}+
\|\nabla u^{(t)}-\nabla v^{(t)}\|_{L^2(\Omega)}^{2(2-\gamma)}\Big)
\end{split}\end{equation*}
and thus, integrating in~$t\in(0,T)$, and exploiting the H\"older Inequality, 
\begin{equation*}\begin{split}&
\| f_u-f_v\|_{L^2((0,T),\,L^2(\Omega))}^2\\
&\qquad\le C
\epsilon \Big(\|u\|_{L^2((0,T),\,H^1(\Omega))}^2+1\Big)\\
&\qquad\qquad  
+\frac{C}\epsilon\Big(\|u-v\|_{L^2((0,T),\,L^2(\Omega))}^{2}+T^{2-\gamma}\|u-v\|_{L^2((0,T),\,L^2(\Omega))}^{2(\gamma-1)}
\\
&\qquad\qquad\qquad\qquad+T^{\gamma-1}\|\nabla u -\nabla v \|_{L^2((0,T),\,L^2(\Omega))}^{2(2-\gamma)}\Big)\\
&\qquad\le
C\epsilon \Big(\|u\|_{X}^2+1\Big)
+\frac{C}\epsilon\Big(\|u-v\|_{X}^{2}+ \|u-v\|_{X}^{2(\gamma-1)}+ \|u- v \|_{X}^{2(2-\gamma)}\Big)\\
&\qquad\le
C\epsilon \Big(\|u\|_{X}^2+1\Big)
+\frac{C}\epsilon\Big( \|u-v\|_{X}^{2(\gamma-1)}+ \|u- v \|_{X}^{2(2-\gamma)}\Big).
\end{split}\end{equation*}

In particular, choosing
$$ \epsilon:=\sqrt{\frac{ \|u-v\|_{X}^{2(\gamma-1)}+ \|u- v \|_{X}^{2(2-\gamma)}
}{ \|u\|_{X}^2+1}},$$
we gather that
\begin{eqnarray*} \| f_u-f_v\|_{L^2((0,T),\,L^2(\Omega))}^2& \le &C
\sqrt{ \|u\|_{X}^2+1} \;\sqrt{ \|u-v\|_{X}^{2(\gamma-1)}+ \|u- v \|_{X}^{2(2-\gamma)}}
\\&
\le& C_{u}\,\sqrt{ \|u-v\|_{X}^{2(\gamma-1)}+ \|u- v \|_{X}^{2(2-\gamma)}}\\&\le& C_{u}\,\Big(\|u-v\|_{X}^{\gamma-1}
+\|u-v\|_{X}^{2-\gamma}\Big).\end{eqnarray*}
Up to renaming constants, this and~\eqref{pol-1} give~\eqref{Lasf.02}, as desired.
\end{proof}

Now we put forth a useful compactness result.

\begin{lemma}\label{LACOMPA}
Let~$v_k\in X$, with
\begin{equation}\label{OMNI}
\sup_{k\in\N}\Big(
\|v_k\|_{L^2((0,T),\,H^2(\Omega))}+\|\partial_t v_k\|_{L^2((0,T),\,L^2(\Omega))}\Big)<+\infty.\end{equation}

Then, there exists~$v\in X$ such that, up to a subsequence, $v_k\to v$ in~$X$ as~$k\to+\infty$.
\end{lemma}

\begin{proof} By~\eqref{OMNI},
up to a subsequence, we can assume that
\begin{equation}\label{PSKJ-1.0.0}
{\mbox{$v_k$ converges weakly to some~$\widetilde v\in L^2((0,T),\,H^1_0(\Omega)\cap H^2(\Omega))$ as~$k\to+\infty$}}
\end{equation}
and, up to extracting a further subsequence, we can assume that
\begin{equation}\label{pskqodc-1}
{\mbox{$v_k$ converges to some~$v$ in~$L^2((0,T),\,L^2(\Omega))$
as~$k\to+\infty$.}}\end{equation}

We claim that
\begin{equation}\label{PSKJ-1.0}
v=\tilde v.
\end{equation}
To this end, in view of~\eqref{PSKJ-1.0.0} we have that, for every~$\phi\in L^2((0,T),\,H^1_0(\Omega)\cap H^4(\Omega))$ with~$\Delta\phi\in L^2((0,T),\, H^1_0(\Omega))$,
{\footnotesize\begin{equation*}\begin{split}&
\int_0^T \Big(\langle \widetilde v(\cdot,t),\phi(\cdot,t)\rangle_{L^2(\Omega)}-
\langle \widetilde v(\cdot,t),\Delta \phi(\cdot,t)\rangle_{L^2(\Omega)}+
\langle \widetilde v(\cdot,t),\Delta^2\phi(\cdot,t)\rangle_{L^2(\Omega)}\Big)\,dt
\\ =&
\int_0^T \left(\langle \widetilde v(\cdot,t),\phi(\cdot,t)\rangle_{L^2(\Omega)}+
\langle \nabla \widetilde v(\cdot,t),\nabla \phi(\cdot,t)\rangle_{L^2(\Omega)}+\sum_{i,j=1}^n
\langle\partial_{x_i x_j}^2 \widetilde v(\cdot,t),\partial_{x_i x_j}^2\phi(\cdot,t)\rangle_{L^2(\Omega)}\right)\,dt\\
=&\lim_{k\to+\infty}
\int_0^T \left(\langle v_k(\cdot,t),\phi(\cdot,t)\rangle_{L^2(\Omega)}+
\langle \nabla v_k(\cdot,t),\nabla \phi(\cdot,t)\rangle_{L^2(\Omega)}+\sum_{i,j=1}^n
\langle\partial_{x_i x_j}^2 v_k(\cdot,t),\partial_{x_i x_j}^2\phi(\cdot,t)\rangle_{L^2(\Omega)}\right)\,dt\\
=&
\lim_{k\to+\infty}
\int_0^T \Big(\langle v_k(\cdot,t),\phi(\cdot,t)\rangle_{L^2(\Omega)}-
\langle v_k(\cdot,t),\Delta \phi(\cdot,t)\rangle_{L^2(\Omega)}+
\langle v_k(\cdot,t),\Delta^2\phi(\cdot,t)\rangle_{L^2(\Omega)}\Big)\,dt.
\end{split}\end{equation*}}
Hence, by~\eqref{pskqodc-1},
\begin{eqnarray*}&&
\int_0^T\Big(\langle \widetilde v(\cdot,t),\phi(\cdot,t)\rangle_{L^2(\Omega)}-
\langle \widetilde v(\cdot,t),\Delta \phi(\cdot,t)\rangle_{L^2(\Omega)}+
\langle \widetilde v(\cdot,t),\Delta^2\phi(\cdot,t)\rangle_{L^2(\Omega)}\Big)\,dt
\\&=&
\int_0^T \Big(\langle v(\cdot,t),\phi(\cdot,t)\rangle_{L^2(\Omega)}-
\langle v(\cdot,t),\Delta \phi(\cdot,t)\rangle_{L^2(\Omega)}+
\langle v(\cdot,t),\Delta^2\phi(\cdot,t)\rangle_{L^2(\Omega)}\Big)\,dt.
\end{eqnarray*}
That is, letting~$w:=v-\widetilde v$,
\begin{equation}\label{KSX21.1}\begin{split}
&\int_0^T \left\langle w(\cdot,t),\Big(\phi(\cdot,t)-\Delta \phi(\cdot,t)+\Delta^2\phi(\cdot,t)\Big)\right\rangle_{L^2(\Omega)}\,dt\\
=\;& \int_0^T \Big(\langle w(\cdot,t),\phi(\cdot,t)\rangle_{L^2(\Omega)}-
\langle w(\cdot,t),\Delta \phi(\cdot,t)\rangle_{L^2(\Omega)}+
\langle w(\cdot,t),\Delta^2\phi(\cdot,t)\rangle_{L^2(\Omega)}\Big)\,dt\\=\;&0.\end{split}\end{equation}

Now we choose~$\phi\in L^2((0,T),\, H^1_0(\Omega)\cap H^2(\Omega))$ as the minimizer of the functional
$$ \iint_{\Omega\times(0,T)}\Big( |\Delta\psi(x,t)|^2+ |\nabla\psi(x,t)|^2+|\psi(x,t)|^2-2w(x,t)\psi(x,t)\Big)\,dx\,dt.$$
We see that~$\phi$ is a weak solution of
$$ 
\begin{cases}
\phi(x,t)-\Delta \phi(x,t)+\Delta^2\phi(x,t)=w(x,t) &
{\mbox{ for $(x,t)\in \Omega\times(0,T)$,}}\\
\phi(x,t)=0 &{\mbox{ for $(x,t)\in(\partial\Omega)\times(0,T)$,}}\\
\Delta\phi(x,t)=0 &{\mbox{ for $(x,t)\in(\partial\Omega)\times(0,T)$.}}
\end{cases}$$
Actually, $\phi(\cdot,t)\in H^4(\Omega)$, see e.g.
Theorem~4 on page~317 of~\cite{MR2597943}, hence the equation is also fulfilled a.e., and thus, substituting into~\eqref{KSX21.1},
$$ \int_0^T\langle w(\cdot,t), w(\cdot,t)\rangle_{L^2(\Omega)}\,dt=0.$$
This gives that~$w=0$, which in turn proves~\eqref{PSKJ-1.0}.

In particular, from~\eqref{OMNI}, \eqref{PSKJ-1.0.0}
and~\eqref{PSKJ-1.0}, if~$w_k:=v_k-v$,
\begin{equation}\label{pskqodc-2}\begin{split}& \sup_{k\in\N}\|w_k\|_{L^2((0,T),\,H^2(\Omega))}=\sup_{k\in\N}\|v_k-v\|_{L^2((0,T),\,H^2(\Omega))}=
\sup_{k\in\N}\|\sigma_{f_{u_k}}-\widetilde v\|_{L^2((0,T),\,H^2(\Omega))}\\&\qquad\le
\sup_{k\in\N}\|\sigma_{f_{u_k}}\|_{L^2((0,T),\,H^2(\Omega))}+\|\widetilde v\|_{L^2((0,T),\,H^2(\Omega))}
<+\infty.\end{split}\end{equation}

Now we use
the Gagliardo-Nirenberg Interpolation Inequality, see e.g. Theorem~1 in~\cite{MR208360}
(used here with~$a:=\frac12$, $j:=1$, $m:=2$, $p:=2$, $r:=2$, and~$q:=2$),
obtaining that
\begin{equation*} \begin{split}&\int_\Omega |\nabla w_k(x,t)|^2\,dx\\&\qquad\le
C\,\left[ \left( \int_\Omega |D^2 w_k(x,t)|^2\,dx\right)^{\frac{1}{2}}
\left( \int_\Omega |w_k(x,t)|^2\,dx\right)^{\frac{1}{2}}+
\int_\Omega |w_k(x,t)|^{2}\,dx
\right],\end{split}\end{equation*}
for some~$C>0$ depending only on~$n$ and~$\Omega$.

Hence, integrating for~$t\in(0,T)$ and using the H\"older Inequality,
\begin{eqnarray*}
\|w_k\|^2_{L^2((0,T),\,H^1(\Omega))}&\le&
C\int_0^T \left( \int_\Omega |D^2 w_k(x,t)|^2\,dx\right)^{\frac{1}{2}}
\left( \int_\Omega |w_k(x,t)|^2\,dx\right)^{\frac{1}{2}}\,dt\\&&\qquad\qquad\qquad+C\|w_k\|^2_{L^2((0,T),\,L^2(\Omega))}\\&\le&
C\sqrt{\int_0^T \left(\int_\Omega |D^2 w_k(x,t)|^2\,dx\right)\,dt\,\int_0^T
\left( \int_\Omega |w_k(x,t)|^2\,dx\right)\,dt}\\&&\qquad\qquad\qquad+C\|w_k\|^2_{L^2((0,T),\,L^2(\Omega))}\\&\le& C\Big( \|w_k\|_{L^2((0,T),\,H^2(\Omega))}\,\|w_k\|_{L^2((0,T),\,L^2(\Omega))}+
\|w_k\|^2_{L^2((0,T),\,L^2(\Omega))}\Big).
\end{eqnarray*}
{F}rom this, \eqref{pskqodc-1}, and~\eqref{pskqodc-2}, we conclude that~$w_k\to0$ (and therefore~$v_k\to v$)
in~$X$.
\end{proof}

\begin{corollary}\label{v:coe}
Let~$\Omega$ be a bounded domain of~$\R^n$ with~$C^2$ boundary.
Let also~$K\in L^2(\Omega\times\Omega)$, $\Theta\in L^\infty(\Omega\times(0,1))$,
$\omega\in L^\infty(\Omega\times(0,1))$, and~$\beta\in W^{1,\infty}(\R)$.

Then, there exists~$T_\star\in(0,1]$, depending only on~$n$, $\Omega$,
$\|K\|_{L^2(\Omega\times\Omega)}$, 
$\|\omega\|_{ L^\infty(\Omega\times(0,1))}$, and~$\|\beta\|_{ W^{1,\infty}(\R)}$, such that if~$T\in(0,T_\star]$ the following claims hold true.

The nonlinear operator~$A$ satisfies~$A:X\to X$.

Also, it is continuous and compact, and the set
\begin{equation}\label{Sela} \big\{ {\mbox{$u\in X$  s.t. $u=\lambda Au$ for some~$\lambda\in[0,1]$}}\big\}\end{equation}
is bounded.
\end{corollary}

\begin{proof} 
Let~$T\in(0,1]$.
By Lemma~\ref{LA23} we know that~$A:X\to X$, and the continuity of~$A$ is a consequence of~\eqref{Lasf.02}.

As for the compactness of~$A$, let~$u_k$ be a bounded sequence in~$X$.
Then, by Lemma~\ref{Lasf}, the function~$f_{u_k}$ is bounded in~$Y$
and, consequently, by~\eqref{REGES},
\begin{equation}\label{PSKJ-1.1} 
\sup_{k\in\N}\|\sigma_{f_{u_k}}\|_{H^1((0,T),\,L^2(\Omega))}+
\sup_{k\in\N}\|\sigma_{f_{u_k}}\|_{L^2((0,T),\,H^2(\Omega))}<+\infty.\end{equation}
On this account, using Lemma~\ref{LACOMPA} with~$v_k:=\sigma_{f_{u_k}}$, we see that~$v_k=\sigma_{f_{u_k}}=Av_k$ converges to~$v$, thus establishing the desired compactness property for the operator~$A$.

Let now~$u$ belong to the set in~\eqref{Sela}. Then, by~\eqref{REGES} and
Lemma~\ref{LE21.1},
\begin{eqnarray*}&& \|u\|_X=\lambda\|Au\|_X\le\|Au\|_X=\|\sigma_{f_u}\|_{L^2((0,T),\, H^1(\Omega))}\\
&&\qquad=\sqrt{\int_0^T \|\sigma_{f_u}(\cdot,t)\|_{H^1(\Omega)}^2\,dt}\le
\sqrt{T\,\sup_{t\in[0,T] }\|\sigma_{f_u}(\cdot,t)\|_{H^1(\Omega)}^2}\\&&\qquad\le
C\sqrt{T}\,\Big(\|f_u\|_{Y}+\|g\|_{H^1(\Omega)}\Big),
\end{eqnarray*}
with~$C>0$ depending only on~$n$ and~$\Omega$.

Hence, by Lemma~\ref{Lasf},
\begin{eqnarray*}&& \|u\|_X\le
C\sqrt{T}\,\left(\|u\|_X+
\|u\|_X^{2-\gamma}+T\sup_{\Omega\times(0,T)} \Theta_-
+\|g\|_{H^1(\Omega)}\right),\end{eqnarray*}
with~$C>0$ now depending only on~$n$, $\Omega$, $\|K\|_{L^2(\Omega\times\Omega)}$,
$\|\omega\|_{ L^\infty(\Omega\times(0,T))}$, and~$\|\beta\|_{W^{1,\infty}(\R)}$.

Now we use Young Inequality with exponents~$\frac1{2-\gamma}$ and~$\frac1{\gamma-1}$ and we see that
$$\|u\|_X^{2-\gamma}\le (2-\gamma)\|u\|_X + \gamma-1\le
2\|u\|_X +1.
$$
Accordingly, up to renaming~$C$,
\begin{eqnarray*}&& \|u\|_X\le
C\sqrt{T}\,\left(\|u\|_X+T\sup_{\Omega\times(0,T)} \Theta_-
+\|g\|_{H^1(\Omega)}+1\right),\end{eqnarray*}

When~$C\sqrt{T}\le\frac12$, we can reabsorb the term~$\|u\|_X$ to the left-hand side, thus obtaining, up to renaming constants,
$$ \|u\|_X\le
C\left(\sup_{\Omega\times(0,T)} \Theta_-+\|g\|_{H^1(\Omega)}+1\right),$$proving that the set in~\eqref{Sela} is bounded.
\end{proof}

With this preliminary work, we can now give the proof of Theorem~\ref{MAIN:THEOREM} by arguing as follows:

\begin{proof}[Proof of Theorem~\ref{MAIN:THEOREM}] By Corollary~\ref{v:coe}, the hypotheses of Schaefer's Fixed-Point Theorem are satisfied (see e.g. Theorem~11.1 in~\cite{MR3967045}). This gives the existence of a function~$u$ satisfying~$u=Au=\sigma_{f_u}$, that is
$$ \partial_t u=\Delta u+f_u,$$
yielding the desired result in view of~\eqref{defu}.
\end{proof}

\section{Proof of Theorem~\ref{MAIN:THEOREM:LIN}}\label{PWSJM-0-i4rpo2438bv.23tno06u65830-1}

The gist of the proof of Theorem~\ref{MAIN:THEOREM:LIN} consists in taking a sequence of parameters~$\gamma\searrow1$
and the corresponding solutions~$u_\gamma$ provided by
Theorem~\ref{MAIN:THEOREM}, then utilizing the parabolic regularity theory to infer uniform bounds and compactness properties
that allow one to pass the solutions to the limit.

The technical details go as follows:

\begin{proof}[Proof of Theorem~\ref{MAIN:THEOREM:LIN}] We use the notation in~\eqref{defu}, writing here~$ f^{(\gamma)}_u$
instead of simply~$f_u$ to make explicit the dependence on~$\gamma$. 

We let~$T_\star$ be given by Theorem~\ref{MAIN:THEOREM}.
Then, in the interval of time~$[0,T_\star]$,
we obtain by Theorem~\ref{MAIN:THEOREM} the existence of a solution~$u_\gamma$ of~$\partial_t u_\gamma=\Delta u_\gamma+f^{(\gamma)}_{u_\gamma}$, and, in view of~\eqref{REGES},
\begin{equation}\label{bff5juk890l20-19oe0iur-Xw-1390-1.1}\begin{split}&
\sup_{t\in[0,T_\star]}\|u_\gamma(\cdot,t)\|_{H^1(\Omega)}+\|u_\gamma\|_{L^2((0,T_\star),\,H^2(\Omega))}+\|\partial_t u_\gamma\|_{L^2((0,T_\star),\,L^2(\Omega))}\\&\qquad\le C\Big(\|f^{(\gamma)}_{u_\gamma}\|_{L^2((0,T_\star),\,L^2(\Omega))}+\|g\|_{H^1(\Omega)}\Big),\end{split}
\end{equation}
for some~$C>0$ depending only on~$n$ and~$\Omega$ (recall Lemma~\ref{LE21.1}).

Furthermore, in light of Lemmata~\ref{LEa-1} and~\ref{LEa-2}, setting~${u^{(t)}}:=u(\cdot,t)$, we have that
\begin{equation*}\begin{split}
\| f^{(\gamma)}_{ u_\gamma}(\cdot,t)\|_{L^2(\Omega)}
\le\,& \sqrt{2} \|K\|_{L^2(\Omega\times\Omega)}\,\left(\|u^{(t)}_\gamma\|_{L^2(\Omega)}+|\Omega|^{\frac12}\sup_{\Omega\times(0,T_\star)} \Theta_-
\right) \\&\quad+  \sqrt{2}\left(
\|\omega(\cdot,t)\|_{L^\infty(\Omega)}\,\|\nabla u^{(t)}_\gamma\|_{L^2(\Omega)}
+\|\beta\|_{L^\infty(\R)}\,|\Omega|^{\frac{\gamma-1}{2}}
\|\nabla u^{(t)}_\gamma\|_{L^2(\Omega)}^{2-\gamma}\right)\\
\le\,& C\Big(\|K\|_{L^2(\Omega\times\Omega)}+\|\omega(\cdot,t)\|_{L^\infty(\Omega)}+\|\beta\|_{L^\infty(\R)}
\Big)\Big(\|u^{(t)}_\gamma\|_{H^1(\Omega)}+1\Big)\\
\le\,&C\epsilon_0\Big(\|u^{(t)}_\gamma\|_{H^1(\Omega)}+1\Big),\end{split}
\end{equation*}
up to renaming~$C>0$, possibly in dependence of~$n$, $\Omega$ and~$ \|\Theta_-\|_{L^\infty(\Omega\times(0,1))}$.

As a result,
$$ \|f^{(\gamma)}_{u_\gamma}\|_{L^2((0,T_\star),\,L^2(\Omega))}\le
C\epsilon_0\Big(\|u_\gamma\|_{L^2((0,T_\star),\,H^1(\Omega))}+1\Big),$$
up to renaming~$C$, which, combined with~\eqref{bff5juk890l20-19oe0iur-Xw-1390-1.1}, gives that
\begin{equation*}\begin{split}&
\sup_{t\in[0,T_\star]}\|u_\gamma(\cdot,t)\|_{H^1(\Omega)}+\|u_\gamma\|_{L^2((0,T_\star),\,H^2(\Omega))}+\|\partial_t u_\gamma\|_{L^2((0,T_\star),\,L^2(\Omega))}\\&\qquad\le C\Big(
\epsilon_0\|u_\gamma\|_{L^2((0,T_\star),\,H^1(\Omega))}+1+\|g\|_{H^1(\Omega)}\Big).\end{split}
\end{equation*}
Consequently, if~$\epsilon_0$ is small enough (possibly in dependence of~$n$, $\Omega$ and~$ \|\Theta_-\|_{L^\infty(\Omega\times(0,1))}$), we can reabsorb one term to the left-hand side and conclude that
\begin{equation}\label{1s2fbv43M0ojfdIS-2}\begin{split}&
\sup_{t\in[0,T_\star]}\|u_\gamma(\cdot,t)\|_{H^1(\Omega)}+\|u_\gamma\|_{L^2((0,T_\star),\,H^2(\Omega))}+\|\partial_t u_\gamma\|_{L^2((0,T_\star),\,L^2(\Omega))}\\&\qquad\qquad \le C\Big(
1+\|g\|_{H^1(\Omega)}\Big).\end{split}
\end{equation}

We now take a sequence~$\gamma_k\searrow1$ and we obtain from~\eqref{1s2fbv43M0ojfdIS-2}
by compactness (see Lemma~\ref{LACOMPA}) that
\begin{equation}\label{131fdsbpaaltu}
{\mbox{$u_{\gamma_k}$
converges, up to a subsequence, to some~$u$ in~$L^2((0,T_\star),\,H^1(\Omega))$.}}\end{equation}

Another consequence of~\eqref{1s2fbv43M0ojfdIS-2} is that
$$ \iint_{\Omega\times(0,T_\star)}|\partial_t u_{\gamma_k}(x,t)|^2\,dx\,dt\le C^2\Big(
1+\|g\|_{H^1(\Omega)}\Big)^2$$
and so, up to a subsequence, 
\begin{equation}\label{131fdsbpaaltu-2}
{\mbox{$\partial_t u_{\gamma_k}$ converges weakly to some function~$U$ in~$L^2(\Omega\times(0,T_\star))$.}}\end{equation}

We claim that, in the weak sense,
\begin{equation}\label{131fdsbpaaltu-3}
U=\partial_t u.
\end{equation}
Indeed, for all~$\psi\in C^\infty_0(\Omega\times(0,T_\star))$, we infer from~\eqref{131fdsbpaaltu} and~\eqref{131fdsbpaaltu-2} that
\begin{eqnarray*}&&
\iint_{\Omega\times(0,T_\star)} U(x,t)\,\psi(x,t)\,dx\,dt=\lim_{k\to+\infty}
\iint_{\Omega\times(0,T_\star)} \partial_t u_{\gamma_k}(x,t)\,\psi(x,t)\,dx\,dt\\&&\qquad=-\lim_{k\to+\infty}
\iint_{\Omega\times(0,T_\star)}u_{\gamma_k}(x,t)\, \partial_t \psi(x,t)\,dx\,dt
=-\iint_{\Omega\times(0,T_\star)}u(x,t)\, \partial_t \psi(x,t)\,dx\,dt,
\end{eqnarray*}
which establishes~\eqref{131fdsbpaaltu-3}.

It remains to show that~$\partial_t u=\Delta u+f^{(1)}_u$ in the weak sense, i.e. (see~\cite[page 352]{MR2597943})
that, for a.e.~$t\in[0,T_\star]$ and all~$\psi\in H^1_0(\Omega)$,
\begin{equation}\label{SLm0irkj3-o3rFHS}
\langle \partial_tu(\cdot,t),\psi\rangle_{H^{-1}(\Omega),\,H^1_0(\Omega)}+\int_\Omega\nabla u(x,t)\cdot\nabla\psi(x)\,dx-\int_\Omega f^{(1)}_u(x,t)\,\psi(x)\,dx=0.\end{equation}
To this end, we let~$w_k:=u-u_{\gamma_k}$ and, in light of~\eqref{131fdsbpaaltu}, we recall that
\begin{equation}\label{LASM09oiujhbv-poijhdbc0okhwdf01X2DJOojn}
\lim_{k\to+\infty}\|w_k\|_{L^2((0,T_\star),\,H^1(\Omega))}=0.
\end{equation}
Moreover, by~\eqref{131fdsbpaaltu-2} and~\eqref{131fdsbpaaltu-3}, as~$k\to+\infty$,
\begin{equation}\label{KPSDLMzsiksmdpwedkjnsI}
{\mbox{$\partial_t w_k\to0$ weakly in~$L^2(\Omega\times(0,T_\star))$.}}
\end{equation}
We also call~$\zeta(t)$ the left-hand side of~\eqref{SLm0irkj3-o3rFHS}
(in this way, to prove~\eqref{SLm0irkj3-o3rFHS}, we will need to show that~$\zeta$ vanishes almost everywhere).

Since
\begin{equation*}
\langle \partial_tu_{\gamma_k}(\cdot,t),\psi\rangle_{H^{-1}(\Omega),\,H^1_0(\Omega)}+\int_\Omega\nabla u_{\gamma_k}(x,t)\cdot\nabla\psi(x)\,dx-\int_\Omega f^{(\gamma_k)}_{u_{\gamma_k}}(x,t)\,\psi(x)\,dx=0,\end{equation*}
we find that
\begin{equation}\label{KPSDLMzsiksmdpwedkjnsI2}\begin{split}
&\zeta(t)=\langle \partial_t w_k(\cdot,t),\psi\rangle_{H^{-1}(\Omega),\,H^1_0(\Omega)}+\int_\Omega\nabla w_k(x,t)\cdot\nabla\psi(x)\,dx\\&\qquad\qquad\qquad-\int_\Omega \Big(f^{(1)}_u(x,t)-f^{(\gamma_k)}_{u_{\gamma_k}}(x,t)\Big)\,\psi(x)\,dx.\end{split}
\end{equation}

Let now make the duality~$\langle\cdot,\cdot\rangle_{H^{-1}(\Omega),\,H^1_0(\Omega)}$
more explicit by using the canonical embedding of~$L^2(\Omega)$ into~$H^{-1}(\Omega)$.
Specifically, one identifies~$L^2(\Omega)$ with its dual, which is contained in the dual of~$H^1_0(\Omega)$,
which is identified with~$H^{-1}(\Omega)$, therefore, for~$h\in L^2(\Omega)\subset H^{-1}(\Omega)$, the pairing~$\langle h,\psi\rangle_{H^{-1}(\Omega),\,H^1_0(\Omega)}$ reduces to~$\int_\Omega h(x)\,\psi(x)\,dx$, which is indeed a bounded linear functional on~$H^1_0(\Omega)$.

Consequently, for every~$\Psi\in C^\infty_0((0,T_\star))$, we deduce from~\eqref{KPSDLMzsiksmdpwedkjnsI} that
\begin{equation*}
\lim_{k\to+\infty}\int_{0}^{T_\star}
\langle \partial_t w_k(\cdot,t),\psi\rangle_{H^{-1}(\Omega),\,H^1_0(\Omega)}\,\Psi(t)\,dt=
\lim_{k\to+\infty}\iint_{\Omega\times(0,T_\star)}\partial_t w_k(x,t)\,\psi(x)\,\Psi(t)\,dx\,dt=0.
\end{equation*}

Plugging this information into~\eqref{KPSDLMzsiksmdpwedkjnsI2} we conclude that, for every~$\Psi\in C^\infty_0((0,T_\star))$,
\begin{equation}\label{LASM09oiujhbv-poijhdbc0okhwdf01X2DJOojn2}\begin{split}
\int_{0}^{T_\star}\zeta(t)\,\Psi(t)\,dt&=\lim_{k\to+\infty}\Bigg(\,
\iint_{\Omega\times(0,T_\star)}
\nabla w_k(x,t)\cdot\nabla\psi(x)\,\Psi(t)\,dx\,dt\\&\qquad\qquad -\iint_{\Omega\times(0,T_\star)} \Big(f^{(1)}_u(x,t)-f^{(\gamma_k)}_{u_{\gamma_k}}(x,t)\Big)\,\psi(x)\,\Psi(t)\,dx\,dt\Bigg).\end{split}
\end{equation}

Also, 
\begin{eqnarray*}&&
\left|\,\iint_{\Omega\times(0,T_\star)}
\nabla w_k(x,t)\cdot\nabla\psi(x)\,\Psi(t)\,dx\,dt\right|^2\\
&&\qquad\le T_\star\,|\Omega|\iint_{\Omega\times(0,T_\star)}
|\nabla w_k(x,t)|^2\,|\nabla\psi(x)|^2\,|\Psi(t)|^2\,dx\,dt\\&&\qquad\le T_\star\,|\Omega|\,\|\nabla\psi\|^2_{L^\infty(\Omega)}\,\|\Psi\|^2_{L^\infty((0,T_\star))}\,\|w_k\|_{L^2((0,T_\star),\,H^1(\Omega))}^2,
\end{eqnarray*}
which, owing to~\eqref{LASM09oiujhbv-poijhdbc0okhwdf01X2DJOojn}, is infinitesimal as~$k\to+\infty$.

As a consequence, \eqref{LASM09oiujhbv-poijhdbc0okhwdf01X2DJOojn2} boils down to
\begin{equation}\label{LASM09oiujhbv-poijhdbc0okhwdf01X2DJOojn3}\begin{split}
&\int_{0}^{T_\star}\zeta(t)\,\Psi(t)\,dt=-\lim_{k\to+\infty}\iint_{\Omega\times(0,T_\star)} \Big(f^{(1)}_u(x,t)-f^{(\gamma_k)}_{u_{\gamma_k}}(x,t)\Big)\,\psi(x)\,\Psi(t)\,dx\,dt.\end{split}
\end{equation}

Now we use the notation in~\eqref{f1u}. We also adopt the notation in~\eqref{df2}, but writing~$ f_{2,u}^{(\gamma)}$ in lieu of~$f_{2,u}$ to emphasize the dependence on~$\gamma$.
In this way, we can write that
\begin{equation}\label{LASM09oiujhbv-poijhdbc0okhwdf01X2DJOojn4} f^{(\gamma)}_u=f_{1,u}+f_{2,u}^{(\gamma)}.\end{equation}

By Lemma~\ref{L14}, we know that
\begin{equation*} \begin{split}\|f_{1,u}(\cdot,t)-f_{1,{u_{\gamma_k}}}(\cdot,t)\|_{L^2(\Omega)}&\le\|K\|_{L^2(\Omega\times\Omega)}
\|u(\cdot,t)-{u_{\gamma_k}}(\cdot,t)\|_{L^2(\Omega)}\\&=\|K\|_{L^2(\Omega\times\Omega)}
\|w_k(\cdot,t)\|_{L^2(\Omega)}\end{split}\end{equation*}
and therefore
\begin{equation*}\begin{split}&
\left|\,\iint_{\Omega\times(0,T_\star)} \Big(f_{1,u}(x,t)-f_{1,u_{\gamma_k}}(x,t)\Big)\,\psi(x)\,\Psi(t)\,dx\,dt\right|^2\\
&\qquad\le T_\star\,|\Omega|\,\|\nabla\psi\|^2_{L^\infty(\Omega)}\,\|\Psi\|^2_{L^\infty((0,T_\star))}
\iint_{\Omega\times(0,T_\star)} \Big(f_{1,u}(x,t)-f_{1,u_{\gamma_k}}(x,t)\Big)^2\,dx\,dt\\
&\qquad= T_\star\,|\Omega|\,\|\nabla\psi\|^2_{L^\infty(\Omega)}\,\|\Psi\|^2_{L^\infty((0,T_\star))}
\int_0^{T_\star}  \|f_{1,u}(\cdot,t)-f_{1,{u_{\gamma_k}}}(\cdot,t)\|_{L^2(\Omega)}^2\,dt
\\
&\qquad\le T_\star\,|\Omega|\,\|\nabla\psi\|^2_{L^\infty(\Omega)}\,\|\Psi\|^2_{L^\infty((0,T_\star))}
\|K\|_{L^2(\Omega\times\Omega)}^2\int_0^{T_\star}  \|w_k(\cdot,t)\|^2_{L^2(\Omega)}\,dt
\\
&\qquad=T_\star\,|\Omega|\,\|\nabla\psi\|^2_{L^\infty(\Omega)}\,\|\Psi\|^2_{L^\infty((0,T_\star))}
\|K\|_{L^2(\Omega\times\Omega)}^2\,
\|w_k\|_{L^2((0,T_\star),\,L^2(\Omega))}^2,
\end{split}\end{equation*}
which is infinitesimal as~$k\to+\infty$, thanks to~\eqref{LASM09oiujhbv-poijhdbc0okhwdf01X2DJOojn}.

By virtue of this observation and~\eqref{LASM09oiujhbv-poijhdbc0okhwdf01X2DJOojn4}, we can reduce~\eqref{LASM09oiujhbv-poijhdbc0okhwdf01X2DJOojn3} to
\begin{equation}\label{LASM09oiujhbv-poijhdbc0okhwdf01X2DJOojn5}\begin{split}
&\int_{0}^{T_\star}\zeta(t)\,\Psi(t)\,dt=-\lim_{k\to+\infty}\iint_{\Omega\times(0,T_\star)} \Big(f^{(1)}_{2,u}(x,t)-f^{(\gamma_k)}_{2,u_{\gamma_k}}(x,t)\Big)\,\psi(x)\,\Psi(t)\,dx\,dt.\end{split}
\end{equation}

Now we remark that
{\footnotesize\begin{equation*}\begin{split}&
\Big|f^{(1)}_{2,u}(x,t)-f^{(\gamma_k)}_{2,u_{\gamma_k}}(x,t)\Big|\\
=&\left|
\left( \left(\omega(x,t)+\frac{\beta(u(x,t))\nabla u(x,t)}{|\nabla u(x,t)|}\right)\cdot\nabla u(x,t)\right)_-
-\left( \left(\omega(x,t)+\frac{\beta(u_{\gamma_k}(x,t))\nabla u_{\gamma_k}(x,t)}{|\nabla u_{\gamma_k}(x,t)|^{\gamma_k}}\right)\cdot\nabla u_{\gamma_k}(x,t)\right)_-
\right|\\
\le&\left|
\left(\omega(x,t)+\frac{\beta(u(x,t))\nabla u(x,t)}{|\nabla u(x,t)|}\right)\cdot\nabla u(x,t)- \left(\omega(x,t)+\frac{\beta(u_{\gamma_k}(x,t))\nabla u_{\gamma_k}(x,t)}{|\nabla u_{\gamma_k}(x,t)|^{\gamma_k}}\right)\cdot\nabla u_{\gamma_k}(x,t)
\right|\\
\le&\|\omega\|_{L^\infty(\Omega\times(0,1))}\big|\nabla u(x,t)-\nabla u_{\gamma_k}(x,t)\big|+\Big|
\beta(u(x,t))|\nabla u(x,t)|-\beta(u_{\gamma_k}(x,t))|\nabla u_{\gamma_k}(x,t)|^{2-\gamma_k}
\Big|\\
\le&\|\omega\|_{L^\infty(\Omega\times(0,1))}\big|\nabla w_k(x,t)\big|+\|\beta\|_{L^\infty(\R)}
\Big||\nabla u(x,t)|-|\nabla u_{\gamma_k}(x,t)|^{2-\gamma_k}\Big|\\
&\quad+
\big| \beta(u(x,t))-\beta(u_{\gamma_k}(x,t))\big|\,
|\nabla u_{\gamma_k}(x,t)|^{2-\gamma_k}\\
\le&\|\omega\|_{L^\infty(\Omega\times(0,1))}\big|\nabla w_k(x,t)\big|+\|\beta\|_{L^\infty(\R)}
\Big||\nabla u(x,t)|-|\nabla u(x,t)|^{2-\gamma_k}\Big|\\
&\quad+
\|\beta\|_{L^\infty(\R)}
\Big||\nabla u(x,t)|^{2-\gamma_k}-|\nabla u_{\gamma_k}(x,t)|^{2-\gamma_k}\Big|+\sqrt{2}
\|\beta\|_{W^{1,\infty}(\R)}|u(x,t))-u_{\gamma_k}(x,t)|^{\frac12}\,
|\nabla u_{\gamma_k}(x,t)|^{2-\gamma_k}.
\end{split}\end{equation*}}
Hence, by~\eqref{Pkasmx-162twgdk10edofhigbvuwefgh032iweo},
{\footnotesize\begin{eqnarray*}&&
\Big|f^{(1)}_{2,u}(x,t)-f^{(\gamma_k)}_{2,u_{\gamma_k}}(x,t)\Big|\\&\le&
C\Big(
\big|\nabla w_k(x,t)\big|+
\big||\nabla u(x,t)|-|\nabla u(x,t)|^{2-\gamma_k}\big|+
\big|\nabla w_k(x,t)\big|^{2-\gamma_k}+
|w_k(x,t)|^{\frac12}\,|\nabla u_{\gamma_k}(x,t)|^{2-\gamma_k}\Big),
\end{eqnarray*}}
for some~$C>0$.

As a result,
\begin{eqnarray*}&&\iint_{\Omega\times(0,T_\star)}
\Big|f^{(1)}_{2,u}(x,t)-f^{(\gamma_k)}_{2,u_{\gamma_k}}(x,t)\Big|^2\,dx\,dt\\&\le&
C\Bigg(
\|w_k\|_{L^2((0,T_\star),\,H^1(\Omega))}^2+\iint_{\Omega\times(0,T_\star)}
\big||\nabla u(x,t)|-|\nabla u(x,t)|^{2-\gamma_k}\big|^2\,dx\,dt\\&&\qquad+
\|w_k\|_{L^2((0,T_\star),\,H^1(\Omega))}^{2(2-\gamma_k)}+
\|w_k\|_{L^2((0,T_\star),\,L^2(\Omega))}
\|\nabla u_{\gamma_k}\|^{2(2-\gamma_k)}_{L^2((0,T_\star),\,L^2(\Omega))}\Bigg),
\end{eqnarray*}
which is infinitesimal as~$k\to+\infty$, thanks to~\eqref{1s2fbv43M0ojfdIS-2},
\eqref{LASM09oiujhbv-poijhdbc0okhwdf01X2DJOojn},
and the Dominated Convergence Theorem.

{F}rom this observation and~\eqref{LASM09oiujhbv-poijhdbc0okhwdf01X2DJOojn5} we conclude that
$$ \int_{0}^{T_\star}\zeta(t)\,\Psi(t)\,dt=0.$$
Since~$\Psi\in C^\infty_0((0,T_\star))$ is arbitrary, we conclude that~$\zeta=0$ almost everywhere in~$[0,T_\star]$, as desired.
\end{proof}

\section{Proof of Theorem~\ref{MAIN:THEOREM:GLO}}\label{S:LECADDF2}

\subsection{Toolbox}
The gist to prove the global existence result in Theorem~\ref{MAIN:THEOREM:GLO}
is that the local existence result in Theorem~\ref{MAIN:THEOREM} will provide a solution defined
in a uniform interval of time length~$T_\star$, only depending on~$n$, $\Omega$, $\|K\|_{L^2(\Omega\times\Omega)}$,
$\|\omega\|_{L^\infty(\Omega\times (0,1))}$, and~$\|\beta\|_{W^{1,\infty}(\R)}$, therefore one can iterate
Theorem~\ref{MAIN:THEOREM}, starting the flow at~$u(\cdot,T_\star)$, thus obtaining a solution up to time~$2T_\star$, and so on.

The details to make this idea work are perhaps a bit technical, since the equation under consideration is not standard, therefore, for completeness, we provide the complete argument.

This section collects some ancillary results towards the proof of Theorem~\ref{MAIN:THEOREM:GLO}.
The following is a continuity property of the solution at a given time.

\begin{lemma}
Let~$T_\star$ be as in Theorem~\ref{MAIN:THEOREM}. Let~$T\in(0,T_\star]$ and~$u$ be given by Theorem~\ref{MAIN:THEOREM}. 

Then,
\begin{equation}\label{ojsndk-09ijDvbsPkamsdtrfGSuUNS-1}
\begin{split}&
\sup_{t\in[0,T]}\|u(\cdot,t)\|_{H^1(\Omega)}+\|u\|_{L^2((0,T),\,H^2(\Omega))}+\|\partial_t u\|_{L^2((0,T),\,L^2(\Omega))}
\\&\qquad\le
C\Big(1+\|u(\cdot,0)\|_{H^1(\Omega)}\Big),\end{split}
\end{equation}
for some~$C>0$ depending only on~$n$, $\Omega$, $\|K\|_{L^2(\Omega\times\Omega)}$,
$\|\omega\|_{ L^\infty(\Omega\times(0,T))}$, $\|\beta\|_{W^{1,\infty}(\R)}$, and $\|\Theta_-\|_{L^\infty(\Omega\times(0,T))}$.

Moreover, for all~$\widetilde{T}\in(0,T]$,
\begin{equation}
\label{ojsndk-09ijDvbsPkamsdtrfGSuUNS-11}\lim_{\eta\searrow0} \|u(\cdot, \widetilde T-\eta)-u(\cdot,\widetilde T)\|_{L^2(\Omega)}=0\end{equation}
and, for all~$\widetilde{T}\in[0,T)$,
\begin{equation}
\label{ojsndk-09ijDvbsPkamsdtrfGSuUNS-11BIS}
\lim_{\eta\searrow0} \|u(\cdot, \widetilde T+\eta)-u(\cdot,\widetilde T)\|_{L^2(\Omega)}=0.\end{equation}
\end{lemma}

\begin{proof} In the notation of~\eqref{defu}, we deduce from~\eqref{REGES} that
\begin{equation}\label{ojsndk-09ijDvbsPkamsdtrfGSuUNS-2}
\begin{split}&
\sup_{t\in[0,T]}\|u(\cdot,t)\|_{H^1(\Omega)}+\|u\|_{L^2((0,T),\,H^2(\Omega))}+\|\partial_t u\|_{L^2((0,T),\,L^2(\Omega))}
\\&\qquad\le
C\Big(\|f_u\|_{L^2((0,T),\,L^2(\Omega))}+\|u(\cdot,0)\|_{H^1(\Omega)}\Big),\end{split}
\end{equation}
with~$C>0$ depending only on~$n$ and~$\Omega$ (recall Lemma~\ref{LE21.1}).

Moreover, since the set in~\eqref{Sela} is bounded,
\begin{equation} \label{ojsndk-09ijDvbsPkamsdtrfGSuUNS-w}\|u\|_{L^2((0,T),\,H^1_0(\Omega))}\le
C\left(\sup_{\Omega\times(0,T)} \Theta_-+\|u(\cdot,0)\|_{H^1(\Omega)}+1\right),\end{equation}
with~$C>0$ now depending also on~$\|K\|_{L^2(\Omega\times\Omega)}$,
$\|\omega\|_{ L^\infty(\Omega\times(0,T))}$, and~$\|\beta\|_{W^{1,\infty}(\R)}$.

Furthermore, in light of Lemma~\ref{Lasf},
$$\|f_u\|_{L^2((0,T),\,L^2(\Omega))}\le C\,\left(\|u\|_{L^2((0,T),\,H^1_0(\Omega))}+\sup_{\Omega\times(0,T)} \Theta_-\right).$$
Combining this information and~\eqref{ojsndk-09ijDvbsPkamsdtrfGSuUNS-w} we obtain that
$$ \|f_u\|_{L^2((0,T),\,L^2(\Omega))}\le C\Big(1+\|u(\cdot,0)\|_{H^1(\Omega)}\Big),$$
up to renaming~$C$ that now may depend also on~$\|\Theta_-\|_{L^\infty(\Omega\times(0,T))}$.

This in turn, together with~\eqref{ojsndk-09ijDvbsPkamsdtrfGSuUNS-2}, gives~\eqref{ojsndk-09ijDvbsPkamsdtrfGSuUNS-1}, as desired.

Now we prove~\eqref{ojsndk-09ijDvbsPkamsdtrfGSuUNS-11}.
For this, we observe that, for all~$t\in(0,\eta)$,
\begin{eqnarray*}&&
|u(x,\widetilde T-\eta)-u(x,\widetilde T)|^2\le\left(\int_0^\eta |\partial_tu(x,\widetilde T-s)|\,ds\right)^2\le
\eta\int_0^\eta |\partial_tu(x,\widetilde T-s)|^2\,ds.
\end{eqnarray*}
Hence, integrating over~$x\in \Omega$,
\begin{eqnarray*}
\|u(\cdot,\widetilde T-\eta)-u(\cdot,\widetilde T)\|^2_{L^2(\Omega)}\le
\eta\int_0^\eta \|\partial_tu(\cdot,\widetilde T-s)\|^2_{L^2(\Omega)}\,ds.
\end{eqnarray*}
This and~\eqref{ojsndk-09ijDvbsPkamsdtrfGSuUNS-1} give that
\begin{eqnarray*}
\|u(\cdot,\widetilde T-t)-u(\cdot,t)\|^2_{L^2(\Omega)}\le
C\eta\Big(1+\|u(\cdot,0)\|_{H^1(\Omega)}\Big),
\end{eqnarray*}
from which~\eqref{ojsndk-09ijDvbsPkamsdtrfGSuUNS-11} follows.

The proof of~\eqref{ojsndk-09ijDvbsPkamsdtrfGSuUNS-11BIS} is analogous to that of~\eqref{ojsndk-09ijDvbsPkamsdtrfGSuUNS-11}.
\end{proof}

The following result is the core step to iterate the local result in Theorem~\ref{MAIN:THEOREM}
to obtain a global solution in Theorem~\ref{MAIN:THEOREM:GLO}.
The idea is to ``glue together'' two solutions of
\begin{equation}\label{LEMD-e2owr3ifekjws5tdysdjfwv0iocavaftgb}
\begin{cases}
\partial_t u(x,\cdot)=\Delta u(x,\cdot)+\displaystyle
\int_{\Omega}\big(u(y,\cdot)-\Theta(y,\cdot)\big)_+\,K(x,y)\,dy \\
\qquad\qquad\qquad\qquad+\displaystyle
\left( \left(\omega+\frac{\beta(u)\nabla u}{|\nabla u|^\gamma}\right)\cdot\nabla u\right)_-\\
\qquad{\mbox{ for $x\in\Omega$,}}
\\
u(x,\cdot)=0\\ \qquad{\mbox{ for $x\in\partial\Omega$,}}\end{cases}
\end{equation}
whenever the final value of the first solution coincides with the initial value of the second solution.

For this, we let~$T_\star$ be as in Theorem~\ref{MAIN:THEOREM} and~$\widetilde u_0$ be a solution of~\eqref{LEMD-e2owr3ifekjws5tdysdjfwv0iocavaftgb}
in~$[0,T_\star]$ with initial datum~$g_0$, as given by Theorem~\ref{MAIN:THEOREM}.

We now argue recursively to construct a global solution. For all~$k\in\N$ with~$k\ge1$, we set~$g_k:=\widetilde u_{k-1}(\cdot, kT_\star)$
and let~$\widetilde u_k$ be a solution of~\eqref{LEMD-e2owr3ifekjws5tdysdjfwv0iocavaftgb} 
in the time interval~$[kT_\star, (k+1)T_\star]$ with initial datum~$g_k$,
as given by Theorem~\ref{MAIN:THEOREM}.

Now let~$ u_0:= \widetilde u_0$ and, for all~$k\in\N$ with~$k\ge1$,
\begin{equation}\label{defrecuk}
u_k(\cdot,t):=\begin{cases} u_{k-1}(\cdot,t) &{\mbox{ if }}t\in[0,kT_\star],\\
\widetilde u_k(\cdot,t) &{\mbox{ if }}t\in(kT_\star,(k+1)T_\star].
\end{cases}\end{equation}
With this notation, we have the following:

\begin{lemma}\label{PERLITE}
For all~$k\in\N$,
we have that~$u_k$ is a solution of~\eqref{LEMD-e2owr3ifekjws5tdysdjfwv0iocavaftgb} in the time interval~$[0,(k+1)T_\star]$.
\end{lemma}

\begin{proof}
When~$k=0$, we have that~$u_0=\widetilde u_0$ is a solution of~\eqref{LEMD-e2owr3ifekjws5tdysdjfwv0iocavaftgb}
in~$[0,T_\star]$, as desired.

When~$k=1$, we have that
$$ u_1(\cdot,t):=\begin{cases} u_{0}(\cdot,t) &{\mbox{ if }}t\in[0,T_\star],\\
\widetilde u_1(\cdot,t) &{\mbox{ if }}t\in(T_\star,2T_\star],
\end{cases}$$
and we claim that
\begin{equation}\label{desfu79870feryrtj39yYYY768o877}
{\mbox{$u_1$ is a solution of~\eqref{LEMD-e2owr3ifekjws5tdysdjfwv0iocavaftgb} in the time interval~$[0,2T_\star]$.}}
\end{equation}
To establish this,
we let~$\eta>0$ (to be taken as small as we wish here below) and~$\tau_\eta\in C^\infty(\R)$ such that~$\tau_\eta(T_\star-t)=\tau_\eta(T_\star+t)$ for all~$t\in\R$,
$\tau_\eta=0$ in~$(T_\star-\eta,T_\star+\eta)$, $\tau_\eta=1$ in~$\R\setminus(T_\star-2\eta,T_\star+2\eta)$, and~$\|\tau_\eta\|_{L^\infty(\R)}\le\frac2\eta$. Given~$\phi\in C^\infty_0((0,2T_\star))$, we let~$\phi_\eta:=\tau_\eta\phi$ and we remark that
\begin{equation*}
\begin{split}&\left|\int_0^{2T_\star} \phi'(t)\,u_1(x,t)\,dt-\int_0^{2T_\star} \phi'_\eta(t)\,u_1(x,t)\,dt\right|
\\&\qquad=\left|\int_0^{2T_\star} \Big(\phi'(t)\big(1-\tau_\eta(t)\big)-\phi(t)\tau_\eta'(t) \Big)\,u_1(x,t)\,dt\right|\\&\qquad\le
\|\phi'\|_{L^\infty(\R)}
\int_{T_\star-2\eta}^{T_\star+2\eta} |u_1(x,t)|\,dt\\&\qquad\qquad+
\left| 
\int_{T_\star-2\eta}^{T_\star} \phi(t)\tau_\eta'(t) u_0(x,t)\,dt
+\int_{T_\star}^{T_\star+2\eta} \phi(t)\tau_\eta'(t) \widetilde u_1(x,t)\,dt
\right|\\&\qquad\le
2\|\phi'\|_{L^\infty(\R)}\sqrt\eta\,
\sqrt{\int_{T_\star-2\eta}^{T_\star+2\eta} |u_1(x,t)|^2\,dt}\\&\qquad\qquad+
\left| 
\int_0^{2\eta} \phi(T_\star-s)\tau_\eta'(T_\star-s) u_0(x,T_\star-s)\,ds\right.\\
&\qquad\qquad\qquad\left.+\int_{0}^{2\eta} \phi(T_\star+s)\tau_\eta'(T_\star+s) \widetilde u_1(x,T_\star+s)\,ds
\right|\\
&\qquad\le
2\|\phi'\|_{L^\infty(\R)}\sqrt\eta\,
\sqrt{\int_{T_\star-2\eta}^{T_\star+2\eta} |u_1(x,t)|^2\,dt}\\&\qquad\qquad+ 
\frac2\eta\int_0^{2\eta}\Big| \phi(T_\star-s) u_0(x,T_\star-s)-
\phi(T_\star+s)\widetilde u_1(x,T_\star+s)\Big|\,ds\\
&\qquad\le
2\|\phi'\|_{L^\infty(\R)}\sqrt\eta\,
\sqrt{\int_{T_\star-2\eta}^{T_\star+2\eta} |u_1(x,t)|^2\,dt}\\&\qquad\qquad+ 
\frac2\eta\int_0^{2\eta}\Big|\big( \phi(T_\star-s)-\phi(T_\star+s)\big) u_0(x,T_\star-s)\Big|\,ds\\&\qquad\qquad+ 
\frac2\eta\int_0^{2\eta}\Big|\phi(T_\star+s)\big(\widetilde u_1(x,T_\star-s)-\widetilde u_1(x,T_\star+s)\big) 
\Big|\,ds\\&\qquad\le
2\|\phi'\|_{L^\infty(\R)}\sqrt\eta\,
\sqrt{\int_{T_\star-2\eta}^{T_\star+2\eta} |u_1(x,t)|^2\,dt}+ 
8\|\phi'\|_{L^\infty(\R)}\int_0^{2\eta} |u_0(x,T_\star-s)|\,ds\\&\qquad\qquad+ 
\frac{2\|\phi\|_{L^\infty(\R)}}\eta\int_0^{2\eta}
\big| u_0(x,T_\star-s)-\widetilde u_1(x,T_\star+s)\big|\,ds\\
&\qquad\le 
C\sqrt{
\eta\int_{T_\star-2\eta}^{T_\star+2\eta} |u_1(x,t)|^2\,dt
+\frac1\eta\int_0^{2\eta}
\big| u_0(x,T_\star-s)-\widetilde u_1(x,T_\star+s)\big|^2\,ds},
\end{split}
\end{equation*}
for some~$C>0$.

As a result, up to renaming~$C$ line after line,
\begin{equation*}\begin{split}&
\left\| \int_0^{2T_\star} \phi'(t)\,u_1(\cdot,t)\,dt-\int_0^{2T_\star} \phi'_\eta(t)\,u_1(\cdot,t)\,dt\right\|_{L^2(\Omega)}^2\\&\qquad\le
C\int_\Omega\left(
\eta\int_{T_\star-2\eta}^{T_\star+2\eta} |u_1(x,t)|^2\,dt
+\frac1\eta\int_0^{2\eta}
\big| u_0(x,T_\star-s)-\widetilde u_1(x,T_\star+s)\big|^2\,ds
\right)\,dx\\
&\qquad=C\left(
\eta\int_{T_\star-2\eta}^{T_\star+2\eta} \|u_1(\cdot,t)\|^2_{L^2(\Omega)}\,dt
+\frac1\eta\int_0^{2\eta}
\| u_0(\cdot,T_\star-s)-\widetilde u_1(\cdot,T_\star+s)\|^2_{L^2(\Omega)}\,ds
\right)
\end{split}
\end{equation*}
and therefore
\begin{equation*}\begin{split}&\lim_{\eta\searrow0}
\left\| \int_0^{2T_\star} \phi'(t)\,u_1(\cdot,t)\,dt-\int_0^{2T_\star} \phi'_\eta(t)\,u_1(\cdot,t)\,dt\right\|_{L^2(\Omega)}^2\\&\qquad\le
C \lim_{\eta\searrow0}\| u_0(\cdot,T_\star-2\eta)-\widetilde u_1(\cdot,T_\star+2\eta)\|^2_{L^2(\Omega)}\\&\qquad\le
C \lim_{\eta\searrow0}\Big(
\| u_0(\cdot,T_\star-2\eta)- u_0(\cdot,T_\star)\|^2_{L^2(\Omega)}+
\| \widetilde u_1(\cdot,T_\star)-\widetilde u_1(\cdot,T_\star+2\eta)\|^2_{L^2(\Omega)}\Big)
.\end{split}
\end{equation*}
{F}rom this, \eqref{ojsndk-09ijDvbsPkamsdtrfGSuUNS-11} and~\eqref{ojsndk-09ijDvbsPkamsdtrfGSuUNS-11BIS}, it follows that
\begin{equation}\label{ojsndk-09ijDvbsPkamsdtrfGSuUNS-11x23}
\lim_{\eta\searrow0}
\left\| \int_0^{2T_\star} \phi'(t)\,u_1(\cdot,t)\,dt-\int_0^{2T_\star} \phi'_\eta(t)\,u_1(\cdot,t)\,dt\right\|_{L^2(\Omega)}^2=0.\end{equation}

We also notice that, in the weak sense (see e.g.~\cite[page~285]{MR2597943}),
\begin{eqnarray*}
&& \int_0^{2T_\star} \phi'_\eta(t)\,u_1(\cdot,t)\,dt=
\int_0^{T_\star-\eta} \phi'_\eta(t)\,u_1(\cdot,t)\,dt+\int_{T_\star+\eta}^{2T_\star} \phi'_\eta(t)\,u_1(\cdot,t)\,dt\\&&\qquad=
\int_0^{T_\star-\eta} \phi'_\eta(t)\,u_0(\cdot,t)\,dt+\int_{T_\star+\eta}^{2T_\star} \phi'_\eta(t)\,\widetilde u_1(\cdot,t)\,dt\\&&\qquad= -
\int_0^{T_\star-\eta} \phi_\eta(t)\,\partial_t u_0(\cdot,t)\,dt-\int_{T_\star+\eta}^{2T_\star} \phi_\eta(t)\,\partial_t\widetilde u_1(\cdot,t)\,dt
\\&&\qquad=-\int_0^{2T_\star} \phi_\eta(t)\,v(\cdot,t)\,dt,
\end{eqnarray*}
where
$$ v(\cdot,t):=\begin{cases} \partial_tu_0(\cdot,t) &{\mbox{ if }}t\in[0,T_\star],\\
\partial_t\widetilde u_1(\cdot,t) &{\mbox{ if }}t\in(T_\star,2T_\star].
\end{cases}$$
We insert this information into~\eqref{ojsndk-09ijDvbsPkamsdtrfGSuUNS-11x23}, finding that
\begin{equation}\label{ojsndk-09ijDvbsPkamsdtrfGSuUNS-11x24}
\lim_{\eta\searrow0}
\left\| \int_0^{2T_\star} \phi'(t)\,u_1(\cdot,t)\,dt+\int_0^{2T_\star} \phi_\eta(t)\,v(\cdot,t)\,dt\right\|_{L^2(\Omega)}^2=0.\end{equation}

Additionally,
\begin{eqnarray*}&&
\left|\int_0^{2T_\star} \phi_\eta(t)\,v(\cdot,t)\,dt-\int_0^{2T_\star} \phi(t)\,v(\cdot,t)\,dt\right|\le
\|\phi\|_{L^\infty(\R)}\int_{T_\star-2\eta}^{T_\star+2\eta} |v(\cdot,t)|\,dt\\&&\qquad\le
\|\phi\|_{L^\infty(\R)}\sqrt{4\eta \int_{T_\star-2\eta}^{T_\star+2\eta} |v(\cdot,t)|^2\,dt}
\end{eqnarray*}
and consequently
\begin{eqnarray*}&&
\lim_{\eta\searrow0}
\left\| \int_0^{2T_\star} \phi_\eta(t)\,v(\cdot,t)\,dt-\int_0^{2T_\star} \phi(t)\,v(\cdot,t)\,dt\right\|_{L^2(\Omega)}^2
\\&&\qquad\le 4\lim_{\eta\searrow0}\eta\|\phi\|_{L^\infty(\R)}^2\iint_{\Omega\times(T_\star-2\eta,T_\star+2\eta)} |v(x,t)|^2\,dx\,dt\\
&&\qquad\le 4\lim_{\eta\searrow0}\eta\|\phi\|_{L^\infty(\R)}^2
\Big( \|\partial_t u_0\|^2_{L^2((T_\star-2\eta,T_\star),\,L^2(\Omega))}+\|\partial_t \widetilde u_1\|^2_{L^2((T_\star,T_\star+2\eta),\,L^2(\Omega))}\Big)\\&&\qquad=0,
\end{eqnarray*}
thanks to~\eqref{ojsndk-09ijDvbsPkamsdtrfGSuUNS-1}.

This and~\eqref{ojsndk-09ijDvbsPkamsdtrfGSuUNS-11x24} give that
$$ \left\| \int_0^{2T_\star} \phi'(t)\,u_1(\cdot,t)\,dt+\int_0^{2T_\star} \phi(t)\,v(\cdot,t)\,dt\right\|_{L^2(\Omega)}^2=0,$$
whence
$$ \int_0^{2T_\star} \phi'(t)\,u_1(\cdot,t)\,dt=-\int_0^{2T_\star} \phi(t)\,v(\cdot,t)\,dt.$$
This gives that~$v=\partial_tu_1$ in the weak sense (see e.g.~\cite[page~285]{MR2597943}).

Now we 
adopt the short notation in~\eqref{defu} and
use that~$u_0$ is a solution in~$[0,T_\star]$ and~$\widetilde u_1$ is a solution in~$[T_\star,2T_\star]$:
thus (see~\cite[page~352]{MR2597943}), we have that, for a.e.~$t\in[0,2T_\star]$ and all~$\psi\in H^1_0(\Omega)$,
\begin{eqnarray*}
&&\langle \partial_tu_1(\cdot,t),\psi\rangle_{H^{-1}(\Omega),\,H^1_0(\Omega)}+\int_\Omega\nabla u_1(x,t)\cdot\nabla\psi(x)\,dx-\int_\Omega f_{u_1}(x,t)\,\psi(x)\,dx\\&&\qquad=
\langle v(\cdot,t),\psi\rangle_{H^{-1}(\Omega),\,H^1_0(\Omega)}+\int_\Omega\nabla u_1(x,t)\cdot\nabla\psi(x)\,dx-\int_\Omega f_{u_1}(x,t)\,\psi(x)\,dx\\&&\qquad=
\begin{cases}&\displaystyle
\langle \partial_tu_0(\cdot,t),\psi\rangle_{H^{-1}(\Omega),\,H^1_0(\Omega)}+\int_\Omega\nabla u_0(x,t)\cdot\nabla\psi(x)\,dx-\int_\Omega f_{u_0}(x)\,\psi(x)\,dx
\\&\qquad {\mbox{ for a.e. }}t\in[0,T_\star],\\&
\displaystyle
\langle \partial_t\widetilde u_1(\cdot,t),\psi\rangle_{H^{-1}(\Omega),\,H^1_0(\Omega)}+\int_\Omega\nabla \widetilde u_1(x,t)\cdot\nabla\psi(x)\,dx-\int_\Omega f_{\widetilde u_1}(x)\,\psi(x)\,dx\\
& \qquad{\mbox{ for a.e. }}t\in[T_\star,2T_\star],\\
\end{cases}\\&&\qquad=0,
\end{eqnarray*}
yielding the desired result in~\eqref{desfu79870feryrtj39yYYY768o877}.

Recursively one also proves that~$u_k$ is a solution of~\eqref{LEMD-e2owr3ifekjws5tdysdjfwv0iocavaftgb} in the time interval~$[0,(k+1)T_\star]$
for all~$k\ge2$, thus completing the proof of Lemma~\ref{PERLITE}.
\end{proof}

\subsection{The iteration argument} We can now complete the proof of the global in time existence result:

\begin{proof}[Proof of Theorem~\ref{MAIN:THEOREM:GLO}]
Under the assumptions of Theorem~\ref{MAIN:THEOREM:GLO}, one obtains from Theorem~\ref{MAIN:THEOREM}
a solution
for a time interval~$[0,T_\star]$, with~$T_\star>0$ depending only on~$n$, $\Omega$,
$\|K\|_{L^2(\Omega\times\Omega)}$,
$\|\omega\|_{ L^\infty(\Omega\times(0,1))}$, and~$\|\beta\|_{ W^{1,\infty}(\R)}$.

Thus, one defines recursively~$u_k$ as in~\eqref{defrecuk}
and, in light pof Lemma~\ref{PERLITE}, obtains a solution in the time interval~$[0,(k+1)T_\star]$ for all~$k\in\N$ with~$k\ge1$,
as desired.
\end{proof}

\vfill

\end{document}